\RequirePackage{fix-cm}
\RequirePackage{amsmath}

\documentclass{svjour3_no_journal_info}
\usepackage[letterpaper,top=3cm,bottom=3cm,left=3cm,right=3cm,marginparwidth=1.75cm]{geometry}

\smartqed  

\usepackage[english]{babel}
\usepackage[utf8]{inputenc}
\usepackage[T1]{fontenc}
\usepackage{color}
\usepackage{amssymb}
\usepackage{bbm}

\usepackage{float}
\usepackage{tabularx}

\usepackage{graphicx}
\usepackage[colorlinks=true, allcolors=blue]{hyperref}

\usepackage{subfigure}

\usepackage{algorithm}
\usepackage{algorithmic}

\usepackage{multirow}

\DeclareMathOperator*{\argmin}{argmin}
\DeclareMathOperator*{\sign}{sign}

\newcommand{\Min}{\mathop{\mathrm{minimize}}}

\newcommand{\RR}{\mathbb{R}}
\newcommand{\cS}{\mathcal{S}}

\newcommand{\vb}{\mathbf{b}}
\newcommand{\vx}{\mathbf{x}}
\newcommand{\vy}{\mathbf{y}}
\newcommand{\vz}{\mathbf{z}}

\newcommand{\vQ}{\mathbf{Q}}

\newcommand{\vR}{\mathbf{R}}
\newcommand{\vbeta}{\mathbf{\beta}}
\newcommand{\vtheta}{\mathbf{\theta}}

\newtheorem{thm}{Theorem}
\newtheorem{prop}[thm]{Proposition}
\newtheorem{assum}{Assumption}

\newtheorem{rem}{Remark}
\newtheorem{defn}{Definition}

\usepackage{mathtools}
\DeclarePairedDelimiter{\dotp}{\langle}{\rangle}

\graphicspath{{/}{image/}}
\title{Run-and-Inspect Method for Nonconvex Optimization and Global Optimality Bounds for R-Local Minimizers
\thanks{This work of Y. Chen is supported in part by Tsinghua Xuetang Mathematics Program and Top Open Program for his short-term visit to UCLA. The work of Y. Sun and W. Yin is supported in part by NSF grant DMS-1720237 and ONR grant N000141712162.}
}
\titlerunning{Run-and-Inspect Method and R-local minimizers}
\author{Yifan Chen
\and
Yuejiao Sun 
\and 
Wotao Yin}
\institute{Yifan Chen\at Department of Mathematical Sciences, Tsinghua University, Beijing, China.\\
\email{chenyifan14@mails.tsinghua.edu.cn}
\and
          Yuejiao Sun \and Wotao Yin \at 
          Department of Mathematics, University of California, Los Angeles, CA 90095.\\
          \email{ sunyj\,/\,wotaoyin@math.ucla.edu}
}
\date{}

\begin{document}

\maketitle
\begin{abstract}
Many optimization algorithms converge to stationary points. When the underlying problem is nonconvex, they may get trapped at local minimizers and occasionally stagnate near saddle points. We propose the Run-and-Inspect Method, which adds an ``inspect'' phase to existing algorithms that helps escape from non-global stationary points. The inspection samples a set of points in a radius $R$ around the current point. When a sample point yields a sufficient decrease in the objective, we resume an existing algorithm from that point. If no sufficient decrease is found, the current point is called an approximate $R$-local minimizer. We show that an $R$-local minimizer is globally optimal, up to a specific error depending on $R$, if the objective function can be implicitly decomposed into a smooth convex function plus a restricted function that is possibly nonconvex, nonsmooth. Therefore, for such nonconvex objective functions, verifying global optimality is fundamentally easier. For high-dimensional problems, we introduce blockwise inspections to overcome the curse of dimensionality while still maintaining optimality bounds up to a factor equal to the number of blocks.
Our method performs well on a set of artificial and realistic nonconvex problems by coupling with gradient descent, coordinate descent, EM, and prox-linear algorithms.
\end{abstract}

\section{Introduction}
\label{section:introduction}
This paper introduces and analyzes \emph{$R$-local minimizers} in a class of nonconvex optimization and develops a Run-and-Inspect Method to find them.

Consider a possibly nonconvex minimization problem:
\begin{align}
\Min F(\vx) \equiv F(x_1,...,x_s),
\label{eq:minF}
\end{align}
where the variable $\vx \in \mathbb{R}^n $ can be decomposed into $s$ blocks $x_1,...,x_s$, $s\ge 1$. We assume $x_i \in \mathbb{R}^{n_i}$. 

We call a point $\bar{\vx}$ an $R$-local minimizer for some $R>0$ if it attains the minimum of $F$ within the ball with center $\bar{\vx}$ and radius $R$. 

In nonconvex minimization, it is relatively cheap to find a local minimizer but difficult to obtain a global minimizer. For a given $R>0$, the difficulty of finding an $R$-local minimizer lies between those two. Informally, they have the following relationships: for any $R>0$,
\begin{align*}
&F~\text{is convex}~\Rightarrow\\
&~\{\text{local minimizers}\} = \{R\text{-local minimizers}\} = \{\text{global minimizers}\};\\
&F~\text{is nonconvex}~\Rightarrow\\
&~\{\text{local minimizers}\} \supseteq \{R\text{-local minimizers}\} \supseteq \{\text{global minimizers}\}.
\end{align*}



We are interested in nonconvex problems for which the last ``$\supseteq$'' holds with ``=,'' indicating that any $R$-local minimizer (for a sufficiently large $R$) is global. This is possible, for example, if $F$ is the sum of a quadratic function and a sinusoidal oscillation:
\begin{align} 
F(x)=\frac{x^2}{2}+a\sin\left( b\pi(x-\frac{1}{2b})\right)+a,
\label{eqn:onedimension}
\end{align}
where $x\in\RR$ and $a,b\in \RR$. The range of oscillation is specified by amplitude $a$ and frequency $\frac{b}{2}$. We use  $-\frac{1}{2b}$ to shift its phase so that the minimizer of $F$ is $x^*=0$. We also add $a$ to level the minimal objective at $F(x^*)=0$.

An example of \eqref{eqn:onedimension} with $a=0.3$ and $b=3$ is depicted in Figure \ref{fig:onedimension example}.
\begin{figure}[H]
  \centering
  \includegraphics[width=200pt]{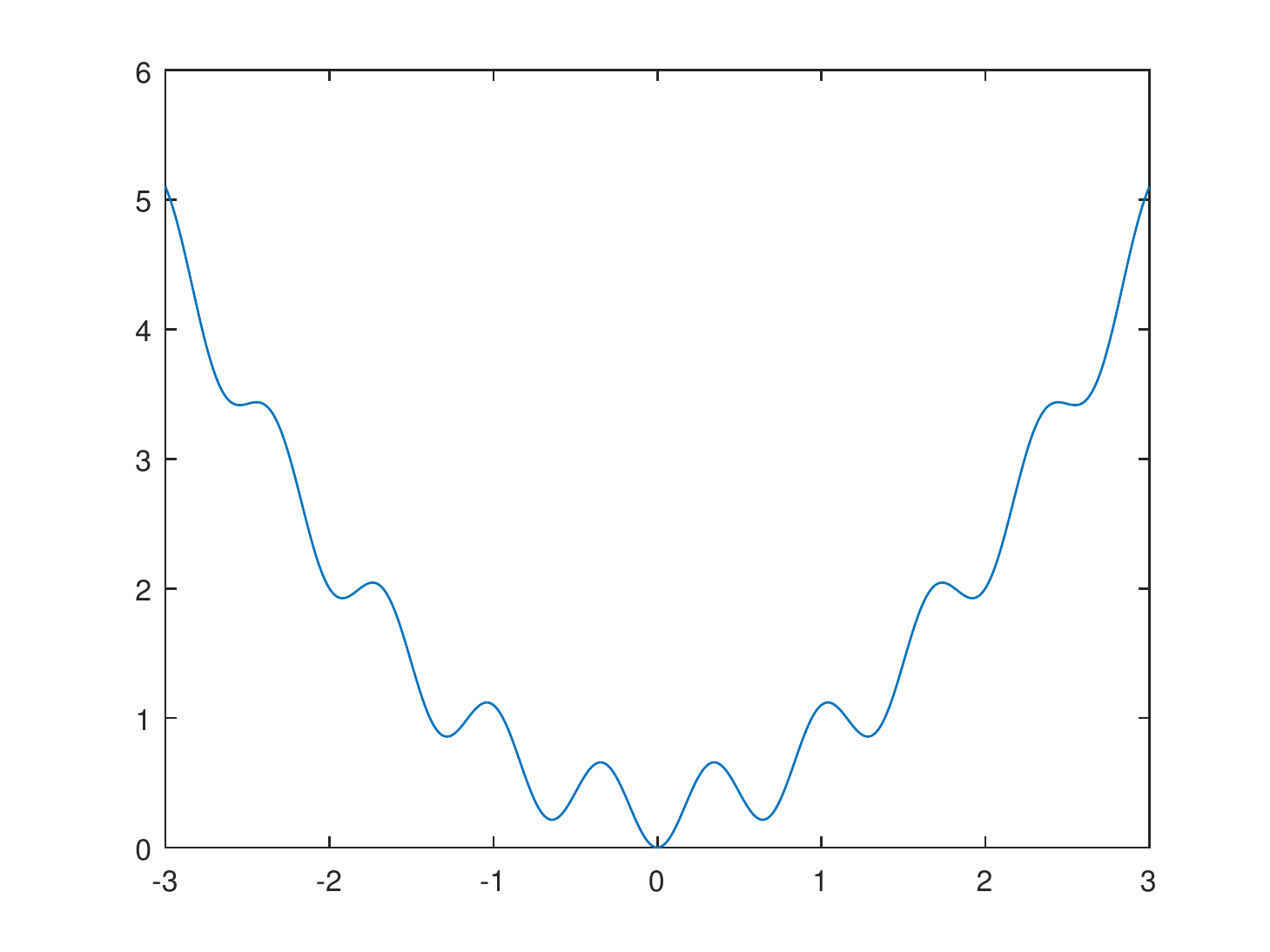}
  \caption{$F(x)$ in \eqref{eqn:onedimension} with $a=0.3,\ b=3$.}
  \label{fig:onedimension example}
\end{figure}
Observe that $F$ has many local minimizers, and its  only global minimizer is $x^*=0$. 
Near each local minimizer $\bar{x}$, we look for an escape point  $x\in[\bar{x}-R,\bar{x}+R]$ such that $f(x) < f(\bar{x})$. We claim that by taking $R\geq \min\{2\sqrt{a},\frac{2}{b}\}$, such an escape point exists for every local minimizer $\bar{x}$ except $\bar{x}=x^*$.

\begin{prop}
Consider minimizing $F$ in \eqref{eqn:onedimension}. If $R\geq\min\{2\sqrt{a},\frac{2}{b}\}$, then the only 
point $\bar{x}$ that satisfies the condition 
\begin{align}\label{eq:Fprop}
F(\bar{x}) = \min\big\{ F(x):{x \in [\bar{x}-R,\bar{x}+R]}\big\}
\end{align}
is the global minimizer $x^*=0$.
\label{prop:one dimensional example}
\end{prop}

\begin{proof}
Suppose $\bar{x} \neq 0$. Without loss of generality we can further assume $\bar{x}>0$. Recall the global minimizer is $x^*=0$. \\
i) If $\bar{x}\leq 2\sqrt{a}$, then $x^*\in [\bar{x}-R,\bar{x}+R]$ gives $F(\bar{x})=0$, so $\bar{x}$ is the global minimizer. Otherwise, we have $F(\bar{x}-2\sqrt{a}) < F(\bar{x}).$ Indeed,
\[F(\bar{x}-2\sqrt{a})-F(\bar{x})\leq \frac{(\bar{x}-2\sqrt{a})^2}{2}-\frac{\bar{x}^2}{2}+2a = 2\sqrt{a}(2\sqrt{a}-\bar{x})<0.\]
However, since $\bar{x}-2\sqrt{a} \in [\bar{x}-R,\bar{x}+R]$, \eqref{eq:Fprop} fails to hold; contradiction. \\
ii) Similar to part i) above, if $\bar{x}\leq \frac{2}{b}$, then $\bar{x}$ is the global minimizer. Otherwise, we have
\[F(\bar{x}-\tfrac{2}{b})-F(\bar{x})= \frac{(\bar{x}-\tfrac{2}{b})^2}{2}-\frac{\bar{x}^2}{2}<0.\]
This leads to the contradiction similar to part i).\qed
\end{proof}


Proposition \ref{prop:one dimensional example} indicates that we can find $x^*$ of this problem 
by locating an approximate local minimizer $\bar{x}^k$ (using a proper algorithm) and then inspecting a small region near $\bar{x}^k$ (e.g., by sampling a set of points). Once the inspection finds a point $x$ such that $f(x)<f(\bar{x}^k)$, resume the algorithm from $x$ and let it find the next approximate local minimizer $\bar{x}^{k+1}$ such that $f(\bar{x}^{k+1})\le f(x)$. Alternate such running and inspection steps until, at a local minimizer $\bar{x}^K$, the inspection fails to find a better point nearby. Then, $\bar{x}^K$ must be an approximate global solution. 
We call this procedure the \emph{Run-and-Inspect Method}. 

The coupling of ``run'' and ``inspect'' is simple and flexible because, no matter which point the ``run'' phase generates, being it a saddle point, local minimizer, or global minimizer, the ``inspect'' phase will either improve upon it or verify its optimality. 
Because saddle points are easier to escape from than a non-global local minimizer, hereafter, we ignore saddle points in our discussion. Related saddle-point avoiding algorithms are reviewed below along with other literature.

Sample-based inspection works in low dimensions. However, it suffers from the curse of dimensionality, as the number of points will increase exponentially with the dimension. For high-dimensional problems, the cost will be prohibitive. To address this issue, we define  the  blockwise $\mathbf{R}$-local minimizer and break the inspection into $s$ blocks of low dimensions: $\vx=[x_1^T~x_2^T \cdots ~x_s^T]^T$ where $x_i \in \mathbb{R}^{n_i}$. We call a point $\bar{\vx}$ a \emph{blockwise $\mathbf{R}$-local minimizer}, where $\mathbf{R}=[R_1~R_2~\cdots~R_s]^T>0$, if it satisfies
\begin{align}
\label{eqn:block R local}
    F(\bar{x}_1,...,\bar{x}_i,...,\bar{x}_s)\leq  \min\limits_{x_i \in B(\bar{x}_i,R_i)} F(\bar{x}_1,...,x_i,...,\bar{x}_s), \ \ \forall 1 \leq i \leq s,
\end{align}
where $B(x,R)$ is a closed ball with center $x$ and radius $R$.
To locate a blockwise $\mathbf{R}$-local minimizer, the inspection is applied to every block while fixing the others. Its cost grows linearly in the number of blocks when the size of every block is fixed. 

This paper studies $R$-local and blockwise $\mathbf{R}$-local minimizers and develop their global optimality bounds
for a class of function $F$ that is the sum of a smooth, strongly convex function and a restricted nonconvex function. (Our analysis assumes a property weaker than strong convexity.) Roughly speaking, the landscape of $F$ is convex at a coarse level, and it can have many local minima. (Arguably, if the landscape of $F$ is overall nonconvex, minimizing $F$ is fundamentally difficult.)

This decomposition is implicit and only used to prove bounds. Our Run-and-Inspect Method, which does \emph{not} use the decomposition, can still provably find a solution that has a bounded distance to a global minimizer and an objective value that is bounded by the global minimum. Both bounds can be zero with a finite $R$. 

The radius $R$ affects theoretical bounds, solution quality, and inspection cost. If $R$ is very small, the inspections will be cheap, but the solution returned by our method will be less likely to be global. On the other hand, an excessive large $R$ leads to expensive inspection and is unnecessary since the goal of inspection is to escape local minima rather than decrease the objective. Theoretically, Theorem~\ref{thm:global case no split} indicates a proper choice $R = 2\sqrt{\beta/L}$, where $\beta,L$ are parameters of the functions in the implicit decomposition. Furthermore, if $R$ is larger than a certain threshold given in Theorem~\ref{thm:global opt global}, then $\bar{\vx}$ returned by our method must be a global minimizer. However, as these value and threshold are associated with the implicit decomposition, they are typically unavailable to the user.

One can imagine that a good practical choice of $R$ would be the radius of the global-minimum valley, assuming this valley is larger than all other local-minimum valleys. This choice is hard to guess, too. Another choice of $R$ is roughly inversely proportional to $\|\nabla f\|$, where $f$ is the smooth convex component in the implicit decomposition of $F$. It is possible to estimate $\|\nabla f\|$ using an area maximum of $\|\nabla F\|$, which itself requires a radius of sampling, unfortunately. ($\|\nabla F\|$ is zero at any local minimizer, so its local value is useless.) However, this result indicates that local minimizers that are far from the global minimizer are easier to escape from. 

We empirically observe that it is both fast and reliable to use a large $R$ and sample the ball $B(\bar{\vx},R)$ outside-in, for example, to sample on a set of rings of radius $R, R-\Delta R, R-2\Delta R,\ldots>0$. In most cases, a point on the first couple of rings is quickly found, and we escape to that point. The smallest ring is almost never sampled except when $\bar{\vx}$ is already an (approximate) global minimizer. Although the final inspection around a global minimizer is generally unavoidable, global minimizers in problems such as compressed sensing and matrix decomposition can be identified without inspection because they have the desired structure or attained a lower bound to the objective value. Anyway, it appears that choosing $R$ is ad hoc but not difficult. Throughout our numerical experiments, we use $R=O(1)$ and obtain excellent results consistently. 

The exposition of this paper is limited to deterministic methods though it is possible to apply stochastic techniques. We can undoubtedly adopt stochastic approximation in the ``run'' phase when, for example, the objective function has a large-sum structure. Also, if the problem has a coordinate-friendly structure~\cite{PengWuXuYanYin2016_coordinate}, we can randomly choose a coordinate, or a block of coordinates, to update each time. Another direction worth pursuing is stochastic sampling during the ``inspect'' phase. These stochastic techniques are attractive in specific settings, but we focus on non-stochastic techniques and global guarantees in this paper.

\subsection{Related work}

\subsubsection{No spurious local minimum}
For certain nonconvex problems, a local minimum is always global or good enough. Examples include tensor decomposition~\cite{ge2015escaping}, matrix completion~\cite{ge2016matrix}, phase retrieval~\cite{sun2016geometric}, and dictionary learning~\cite{sun2015complete} under proper assumptions. When those assumptions are violated to a moderate amount, spurious local minima may appear and be possibly easy to escape. We will inspect them in our future work.

\subsubsection{First-order methods, derivative-free method, and trust-region method}
\label{related work: 1st algorithm}
For nonconvex optimization, there has been recent work on first-order methods that can guarantee convergence to a stationary point. Examples include the block coordinate update method~\cite{xu2013block}, ADMM for nonconvex optimization~\cite{wang2015global}, the accelerated gradient algorithm~\cite{ghadimi2016accelerated}, the stochastic variance reduction method~\cite{reddi2016stochastic}, and so on.

Because the ``inspect'' phase of our method uses a radius, it is seemingly related to the trust-region method~\cite{conn2000trust,martinez2017cubic} and derivative-free method~\cite{ConnScheinbergVicente2009_introduction}, both of which also use a radius at each step. However, the latter methods are not specifically designed to escape from a non-global local minimizer.
\subsubsection{Avoiding saddle points}
A recent line of work aims to avoid saddle points and converge to an $\epsilon$-second-order stationary point $\bar{\vx}$ that satisfies
\begin{align}
\|\nabla F(\bar{\vx})\|\leq \epsilon \quad \text{and} \quad \lambda_{\min}(\nabla^2 F(\bar{\vx}))\geq -\sqrt{\rho \epsilon}, \label{eps-saddle point}
\end{align}
where $\rho$ is the Lipschitz constant of $\nabla^2 F(\vx)$. 
Their assumption is the \emph{strict saddle} property, that is, a point satisfying \eqref{eps-saddle point} for some $\rho>0$ and $\epsilon>0$ must be an approximate local minimizer. 
On the algorithmic side, there are second-order algorithms~\cite{nesterov2006cubic,pascanu2014saddle} and first-order stochastic methods~\cite{ge2015escaping,jin2017escape,panageas2016gradient} that can escape saddle points. The second-order algorithms use Hessian information and thus are more expensive at each iteration in high dimensions. Our method can also avoid saddle points. 


\subsubsection{Simulated annealing}
Simulated annealing (SA)~\cite{kirkpatrick1987optimization} is a classical method in global optimization, and thermodynamic
principles can interpret it. SA uses a Markov chain with a stationary distribution  $\sim e^{-\frac{F(\vx)}{T}}$, where $T$ is the temperature parameter. By decreasing $T$, the distribution tends to concentrate on the global minimizer of $F(\vx)$. However, it is difficult to know exactly when it converges, and the convergence rate can be extremely slow. 

SA can be also viewed as a method that samples the Gibbs distribution using Markov-Chain Monte Carlo (MCMC). Hence, we can apply SA in the ``inspection'' of our method. SA will generate more samples in a preferred area that are more likely to contain a better point, which once found will stop the inspection. 
Apparently, because of the hit-and-run nature of our inspection, we do not need to wait for the SA dynamic to converge. 

\subsubsection{Flat minima in the training of neural networks}
Training a (deep) neural network involves nonconvex optimization. We do not necessarily need to find a global minimizer. A local minimizer will suffice if it generalizes well to data not used in training. There are many recent attempts~\cite{chaudhari2016entropy,chaudhari2017deep,Sagun2016Singularity} that investigate the optimization landscapes and propose methods 
to find
local minima sitting in ``rather flat valleys.'' 

Paper~\cite{chaudhari2016entropy} uses entropy-SGD iteration to favor flatter minima. It can be seen as a PDE-based smoothing technique~\cite{chaudhari2017deep}, which shows that the optimization landscape becomes flatter after smoothing. 
It makes the theoretical analysis easier and provides explanations for many interesting phenomena in deep neural networks. But, as~\cite{wu2017towards} has suggested, a better non-local quantity is required to go further.

\subsection{Notation}
Throughout the paper, $\|\cdot\|$ denotes the Euclidean norm. Boldface lower-case letters (e.g., $\vx$) denote vectors. However, when a vector is a block in a larger vector, it is represented with a lower-case letter with a subscript, e.g., $x_i$. 

\subsection{Organization}
The rest of this paper is organized as follows. 
Section \ref{section:main results} presents the main analysis of $R$-local and blockwise $\mathbf{R}$-local minimizers, and then introduces the Run-and-Inspect Method. Section \ref{section:numerical experiments} presents numerical results of our Run-and-Inspect method. Finally, Section \ref{section conclusions} concludes this paper.

\section{Main Results}
\label{section:main results}
In sections \ref{section:error bound}--\ref{section:blockwise R-local}, we develop theoretical guarantees for our $R$-local and $\mathbf{R}$-local minimizers for a class of nonconvex problems. Then, in section \ref{section:algorithm}, we design algorithms to find those minimizers.

\subsection{Global optimality bounds}
\label{section:error bound}
In this section, we investigate an approach toward deriving error bounds for a point with certain properties. 

Consider problem \eqref{eq:minF}, and let $\vx^*$ denote one of its global minimizers. A global minimizer owns many nice properties. Finding a global minimizer is equivalent to finding a point satisfying all these properties. Clearly, it is easier to develop algorithms that aim at finding a point $\bar{\vx}$ satisfying only some of those properties. An example is that when $F$ is everywhere differentiable, $\nabla F(\vx^*)=0$ is a necessary optimality condition. So, many first-order algorithms that produce a sequence $\vx^k$ such that $\|\nabla F(\vx^k)\|\to 0$ \emph{may} converge to a global minimizer.
Below, we focus on choosing the properties of $\vx^*$ so that a point $\bar{\vx}$ satisfying the same properties will enjoy bounds on $F(\bar{\vx})-F(\vx^*)$ and $\|\bar{\vx}-\vx^*\|$. Of course, proper assumptions on $F$ are needed, which we will make as we proceed.

Let us use $\mathbf{Q}$ to represent a certain set of properties of $\vx^*$, and define
\begin{align}
S_{\mathbf{Q}}=\{\vx : \vx \ \text{satisfies property}\  \vQ\},
\end{align}
which includes $\vx^*$.
For any point $\bar{\vx}$ that also belongs to the set, we have 
\[
F(\bar{\vx})-F(\vx^*) \leq \max_{\vx,\vy \in S_\vQ} F(\vx)-F(\vy) 
\]
and 
\begin{align*}
\|\bar{\vx}-\vx^*\| \leq \text{diam}(S_{\vQ}), 
\end{align*}
where diam($S_{\vQ}$) stands for the diameter of the set $S_{\vQ}$. Hence, the problem of constructing an error bound reduces to analyzing the set $S_{\vQ}$ under certain assumptions on $F$. 

As an example, consider a $\mu$-strongly convex and differentiable $F$ and a simple choice of
$
\mathbf{Q}
$
as $\|\nabla F(\vx)\|\leq \delta$ with $\delta>0$. This choice is admissible since $\|\nabla F(\vx^*)\|=0\leq \delta$.
For this choice, we have
\[
F(\bar{\vx})-F(\vx^*) \leq \frac{\|\nabla F(\bar{\vx})\|^2}{2\mu} \leq \frac{\delta^2}{2\mu}, 
\]
and 
\[
\|\bar{\vx}-\vx^*\| \leq \frac{\|\nabla F(\bar{\vx})\|}{\mu}\leq \frac{\delta}{\mu},
\]
where the first ``$\le$'' in both bounds follows from the strong convexity of $F$. 

We now restrict $F$ to the implicit decomposition 
\begin{align}\label{eq:F=f+r}
\boxed{F(\vx) = f(\vx)+r(\vx).}
\end{align}
We use the term ``implicit'' because this decomposition is only used for analysis, not required by our Run-and-Inspect Method.
Define the sets of the global minimizers of $F$ and $f$ as, respectively,
\begin{align*}
\chi^* &:= \{\vx: F(\vx) = \min_{\vy} F(\vy)\},\\
\chi_f^* &:= \{\vx: f(\vx) = \min_{\vy} f(\vy)\}.
\end{align*}

Below we make three assumptions on \eqref{eq:F=f+r}. The first and third assumptions are used throughout this section. Only some of our results require the second assumption.
\begin{assum}
$f(\vx)$ is differentiable, and $\nabla f(\vx)$ is $L$-Lipschitz continuous. 
\label{assumption:f}
\end{assum}

\begin{assum}
$f(\vx)$ satisfies the Polyak-\L ojasiewicz (PL) inequality \cite{polyak1963gradient} with $\mu>0$:
\begin{align}\label{eqn:PL}
\frac{1}{2}\|\nabla f(\vx)\|^2\geq \mu(f(\vx)-f(\vx^*)), \quad \forall \vx \in\RR^n,~ \vx^*\in\chi_f^*.
\end{align}
\label{assumption:PL inequality}
\end{assum}
Given a point $\vx$, we define its projection
\begin{align*}
\vx_{\text{P}} := \argmin_{\vx^*\in\chi_f^*}\{\|\vx^*-\vx\|
\}.
\end{align*} Then, the PL inequality \eqref{eqn:PL} yields the quadratic growth (QG) condition \cite{Karimi2016Linear}:
\begin{equation}
\label{eqn:QG}
f(\vx)-f(\vx^*)=f(\vx)-f(\vx_{\text{P}})\geq\frac{\mu}{2}\|\vx-\vx_{\text{P}}\|^2,\quad \forall \vx \in\RR^n.
\end{equation}
Clearly, \eqref{eqn:PL} and \eqref{eqn:QG} together imply
\begin{align}\label{eqn:PL+QG}
\|\nabla f(\vx)\|\geq\mu\|\vx-\vx_{\text{P}}\|.
\end{align}
Assumption \ref{assumption:PL inequality} ensures that the gradient of $f$ bounds its objective error.
\begin{assum} $r(\vx)$ satisfies $|r(\vx)-r(\mathbf{y})|\leq \alpha \|\vx-\mathbf{y}\|+2\beta$ in which $\alpha, \beta\ge 0$ are constants.
\label{assumption:r}
\end{assum}
Assumption \ref{assumption:r} implies that $r$ is overall $\alpha$-Lipschitz continuous with additional oscillations up to $2\beta$. In the implicit decomposition \eqref{eq:F=f+r}, though $r$ can cause $F$ to have non-global local minimizers, its impact on the overall landscape of $f$ is limited. For example, the $\ell_p^p ~(0<p<1)$ penalty in compressed sensing is used to induce sparsity of solutions. It is nonconvex and satisfies our assumption
\begin{align*}
    ||x|^p-|y|^p|\leq\big||x|-|y|\big|^p\leq p|x-y|+1-p,\quad x,y\in\RR.
\end{align*}
In fact, many sparsity-induced penalties satisfy Assumption~\ref{assumption:r}. Many of them are sharp near $0$ and thus not Lipschitz there. In Assumption~\ref{assumption:r}, $\beta$ models their variation near $0$ and $\alpha$ controls their increase elsewhere.


In section \ref{section:R-local}, we will show that every $\vx^* \in \chi^*$ satisfies $\|\nabla f(\vx^*)\|\leq \delta$ for a universal $\delta$ that depends on $\alpha,\beta,L$. So, we choose the condition
\begin{align}\label{eq:Q}
\boxed{\mathbf{Q}:~\|\nabla f(\vx)\|\leq \delta.}
\end{align}

To derive the error bound, we introduce yet another assumption:
\begin{assum}
The set $\chi^*_f$ is bounded. That is, there exists $M\geq 0$ such that, for any $\vx,\mathbf{y} \in \chi^*_f$, we have $\|\vx-\mathbf{y}\|\leq M$. 
\label{assumption:minimizer set}
\end{assum}
When $f$ has a unique global minimizer, we have $M=0$ in Assumption \ref{assumption:minimizer set}.

\begin{thm}
Take Assumptions \ref{assumption:f}, \ref{assumption:PL inequality} and \ref{assumption:r}, and assume that all points in $\chi^*$ have property $\mathbf{Q}$ in \eqref{eq:Q}. Then, the following properties hold for every $\bar{\vx}\in S_{\vQ}$\emph{:}
\begin{enumerate}
        \item $F(\bar{\vx})-F^* \leq \frac{\delta^2}{2\mu}+2\beta$, if $\alpha=0$ in Assumption \ref{assumption:r};
        \item  
        $d(\bar{\vx},\chi^*)\leq \frac{2\delta}{\mu}+M$ and $F(\bar{\vx})-F^* \leq \frac{\delta^2+2\alpha \delta}{\mu}+\alpha M + 2\beta$, if $\alpha\geq 0$ and Assumption \ref{assumption:minimizer set} holds. 
\end{enumerate}
\label{thm:error bound}
\end{thm}
\begin{proof}
To show part 1, we have
\begin{align*}
F(\bar{\vx})-F^*&=(f(\bar{\vx})-f(\vx^*))+(r(\bar{\vx})-r(\vx^*))\leq \max\limits_{\vx \in \mathbb{R}^n} \left(f(\bar{\vx})-f(\vx)\right) + 2\beta \\
&\overset{\eqref{eqn:PL}}{\leq} \frac{\|\nabla f(\bar{\vx})\|^2}{2\mu}+2\beta \leq \frac{\delta^2}{2\mu}+2\beta.
\end{align*}
Part 2: Since $f$ satisfies the PL inequality \eqref{eqn:PL} and $\bar{\vx}\in S_{\vQ}$, we have
\[d(\bar{\vx},\chi_f^*) \overset{\eqref{eqn:PL+QG}}{\leq} \frac{\|\nabla f(\bar{\vx})\|}{\mu} \overset{\eqref{eq:Q}}{\leq} \frac{\delta}{\mu}.\]
By choosing an $\vx^* \in \chi^*$ and noticing $\vx^* \in S_{\vQ}$, we also have $d(\vx^*,\chi_f^*) \leq \frac{\delta}{\mu}$ and thus
\[d(\bar{\vx},\chi^*) \leq d(\bar{\vx},\chi_f^*)+M+d(x^*,\chi_f^*)\leq \frac{2\delta}{\mu}+M.\]
Below we let $\bar{\vx}_{\text{P}}$ and $\vx^*_{\text{P}}$, respectively, denote the projections of $\bar{\vx}$ and $\vx^*$ onto the set $\chi^*_f$. Since $f(\bar{\vx}_{\text{P}})=f(\vx^*_{\text{P}})$, we obtain
\begin{align*}
F(\bar{\vx})-F^* & = \left(f(\bar{\vx})-f(\bar{\vx}_{\text{P}})\right)+\left(f(\vx^*_{\text{P}})-f(\vx^*)\right) + \left( r(\bar{\vx})-r(\vx^*)\right) \\
& \leq \frac{\|\nabla f(\bar{\vx})\|^2}{2\mu} + \frac{\|\nabla f(\vx^*)\|^2}{2\mu}+\alpha \|\bar{\vx}-\vx^*\|+2\beta \\
& \leq \frac{\delta^2+2\alpha \delta}{\mu}+\alpha M + 2\beta. &\text{\qed}
\end{align*}
\end{proof}
In the theorem above, we have constructed global optimality bounds for $\bar{\vx}$ obeying $\mathbf{Q}$. 
In the next two subsections, we show that $R$-local minimizers, which include global minimizers, do obey $\mathbf{Q}$ under mild conditions. Hence, the bounds apply to any $R$-local minimizer. 

\subsection{R-local minimizers}
\label{section:R-local}
In this section, we define and analyze $R$-local minimizers. 
We discuss its blockwise version in section \ref{section:blockwise R-local}. Throughout this subsection, we assume that $R \in (0,\infty]$, and $B(\vx,R)$ is a \emph{closed} ball centered at $\vx$ with radius $R$.
\begin{defn}
The point $\bar{\vx}$ is called a standard $R$-local minimizer of $F$ if it satisfies
\begin{align}
F(\bar{\vx}) = \min\limits_{\vx \in B(\bar{\vx},R)} F(\vx).
\label{eq:define R local min}
\end{align}
\end{defn}
Obviously an $R$-local minimizer is a local minimizer, and when $R=\infty$, it is a global minimizer. Conversely, a global minimizer is always an $R$-local minimizer.


We first bound the gradient of $f$ at an $R$-local minimizer so that $\vQ$ in \eqref{eq:Q} is satisfied. 

\begin{thm}
Suppose, in \eqref{eq:F=f+r}, $f$ and $r$ satisfy Assumptions \ref{assumption:f} and \ref{assumption:r}. Then, a point $\bar{\vx}$ obeys $\vQ$ in \eqref{eq:Q} with $\delta$ given in the following two cases: 
\begin{enumerate}
        \item $\delta=\alpha$ if $r$ is differentiable with $\alpha\ge 0$ and $\beta = 0$ in \eqref{assumption:r} and $\bar{\vx}$ is a stationary point of $F$;
        \item $\delta=\alpha + \max\{\frac{4\beta}{R},2\sqrt{\beta L}\}$ if $\bar{\vx}$ is a standard $R$-local minimizer of $F$. 
\end{enumerate}
\label{thm:global case no split}
\end{thm}


\begin{proof}
Under the conditions in part 1, we have $\beta = 0$ and $\nabla F(\bar{\vx})=0$, so $\| \nabla f(\bar{\vx})\|=\|\nabla r(\bar{\vx})\| \leq \alpha=\delta$. Hence, $\vQ$ is satisfied. 

Under the conditions in part 2, $\bar{\vx}$ is an $R$-local minimizer of $F$; hence,
\begin{align}
&\min\limits_{\vx \in B(\bar{\vx},R)} \{f(\vx)-f(\bar{\vx})+r(\vx)-r(\bar{\vx})\}\geq 0 ,\notag \\
\overset{a)}{\Rightarrow} 
& \min\limits_{\vx \in B(\bar{\vx},R)} \{2\beta+ \alpha\|\vx-\bar{\vx} \|  +\dotp{\nabla f(\bar{\vx}),\vx-\bar{\vx}}+\frac{L}{2}\|\vx-\bar{\vx}\|^2\}\geq 0, \notag \\
\overset{b)}{\Leftrightarrow} 
& \min\limits_{\vx \in B(\bar{\vx},R)} \{2\beta+(\alpha-\|\nabla f(\bar{\vx})\|)\|\vx-\bar{\vx}\| + \frac{L}{2}\|\vx-\bar{\vx}\|^2\}\geq 0, \label{eqn:proof inequality condition}
\end{align}
where $a)$ is due to the assumption on $r$ and that $\nabla{f(\vx)}$ is $L$-Lipschitz continuous;
$b)$ is because, as $\|\vx-\bar{\vx}\|$ is fixed, the minimum is attained with $\frac{\vx-\bar{\vx}}{\|\vx-\bar{\vx}\|}=-\frac{\nabla f(\vx)}{\|\nabla f(\vx)\|}$.
If $\|\nabla f(x)\|\le\alpha$, $\vQ$ is immediately satisfied. Now assume $\|\nabla f(x)\|>\alpha$. To simplify \eqref{eqn:proof inequality condition}, we only need to minimize a quadratic function of $\|\vx-\bar{\vx}\|$ over $[0,R]$. Hence, the objective equals
\begin{align}\label{eqn:LHS}
\begin{cases}
 2\beta+(\alpha-\|\nabla f(\bar{\vx})\|)R+\frac{L}{2}R^2, &\quad \text{if~}R\leq\frac{\|\nabla f(\bar{\vx})\|-\alpha}{L},\\
 2\beta-\frac{(\|\nabla f(\bar{\vx})\|-\alpha)^2}{2L}, &\quad \text{otherwise.}
\end{cases}
\end{align}
If $R\leq\frac{\|\nabla f(\bar{\vx})\|-\alpha}{L}$, from $2\beta+(\alpha-\|\nabla f(\bar{\vx})\|)R+\frac{L}{2}R^2\geq 0$, we get
\begin{align*}
&\|\nabla f(\bar{\vx})\|\leq \alpha+\frac{2\beta}{R}+\frac{LR}{2}\leq \alpha+\frac{2\beta}{R}+\frac{\|\nabla f(\bar{\vx})\|-\alpha}{2}\\
\Rightarrow~&\|\nabla f(\bar{\vx})\|\leq \alpha+\frac{4\beta}{R}.
\end{align*}
Otherwise, from $2\beta-\frac{(\|\nabla f(\bar{\vx})\|-\alpha)^2}{2L}\geq 0$, we get
\begin{align*}
 \|\nabla f(\bar{\vx})\|\leq\alpha+2\sqrt{\beta L}.
\end{align*}
Combining both cases, we have $\|\nabla f(\bar{\vx})\|\leq\alpha+\max\{\frac{4\beta}{R},2\sqrt{\beta L}\}=\delta$ and thus $\vQ$.
\qed
\end{proof}
The next result is a consequence of  part 2 of the theorem above. It presents the values of $R$ that ensure the escape from a non-global local minimizer. In addition, more distant local minimizers $\vx$ are easier to escape  in the sense that $R$ is roughly inversely proportional to $\|\nabla f(\vx)\|$.
\begin{corollary}\label{cor:global}
Let $\vx$ be a local minimizer of $F$ and $\|\nabla f(\vx)\|>\alpha+2\sqrt{\beta L}$. As long as either $R>\frac{4\beta}{\|\nabla f(\vx)\|-\alpha}$ or $R\ge 2\sqrt{\beta/L}$, there exists $\vy\in B(\vx,R)$ such that $F(\vy)<F(\vx)$. 
\end{corollary}
\begin{proof}
Assume that the result does \emph{not} hold. Then, $\vx$ is an $R$-local minimizer of $F$. Applying part 2 of Theorem \ref{thm:global case no split}, we get $\|\nabla F(\vx)\|\le \delta = \alpha + \max\{\frac{4\beta}{R},2\sqrt{\beta L}\}$. Combining this with the assumption $\|\nabla f(\vx)\|>\alpha+2\sqrt{\beta L}$, we obtain
\begin{align*}
\alpha+2\sqrt{\beta L} < \|\nabla f(\vx)\| \le \alpha + \max\{\frac{4\beta}{R},2\sqrt{\beta L}\},
\end{align*}
from which we conclude $2\sqrt{\beta L}< \frac{4\beta}{R}$ and $\|\nabla f(\vx)\| \le \alpha + \frac{4\beta}{R}$; We have reached a contradiction.
\qed
\end{proof}

We can further increase $R$ to ensure that any $R$-local minimizer $\bar{\vx}$ is a global minimizer.
\begin{thm}\label{thm:global opt global}
Under Assumptions \ref{assumption:f}, \ref{assumption:PL inequality} and \ref{assumption:r} and $R\ge 2\sqrt{\beta/L}$, we have $d(\bar{\vx},\chi^*) \le 2\tfrac{\alpha+2\sqrt{\beta L}}{\mu}+M$ for any $R$-local minimizer $\bar{\vx}$. 
Therefore, if $R\ge 2\tfrac{\alpha+2\sqrt{\beta L}}{\mu}+M$, any $R$-local minimizer $\bar{\vx}$ is a global minimizer.
\end{thm}
\begin{proof}
According to Theorem~\ref{thm:error bound}, part 2, and Theorem~\ref{thm:global case no split}, part 2,
\begin{align*}
  d(\bar{\vx},\chi^*)\leq 2\frac{\delta}{\mu}+M\leq 2\frac{\alpha+2\sqrt{\beta L}}{\mu}+M,
\end{align*}
where, for the latter inequality, we have used $R\ge 2\sqrt{\beta/L}$ and thus $\max\{4\beta/R, 2\sqrt{\beta L}\} = 2\sqrt{\beta L}$. By convex analysis on $f$, we have $\mu\le L$. Using it with $\alpha\ge 0$ and $M\ge 0$, we further get $2\tfrac{\alpha+2\sqrt{\beta L}}{\mu}+M\ge 4\sqrt{\beta L}/\mu \ge 4\sqrt{\beta L}/L \ge 2\sqrt{\beta/L}$. Therefore, if $R\ge 2\tfrac{\alpha+2\sqrt{\beta L}}{\mu}+M$, then there exists $\vx^*\in\chi^*$ such that $\vx^*\in B(\bar{\vx},R)$. Being an $R$-local minimizer means $\bar{\vx}$ satisfies $F(\bar{\vx})\le F(\vx^*)$, so $\bar{\vx}$ is a global minimizer.
\qed
\end{proof}
\begin{rem}
Since the decomposition \eqref{eq:F=f+r} is implicit, the constants in our analysis are difficult to estimate in practice. However, if we have a rough estimate of the distance between the global minimizer and its nearby local minimizers, then this distance appears to be a good empirical choice for $R$.
\end{rem}

\subsection{Blockwise $\vR$-local minimizers}
\label{section:blockwise R-local}
In this section, we focus on problem \eqref{eq:minF}. This blockwise structure of $F$ motivates us to consider blockwise algorithms. Suppose $\mathbf{R} \in \mathbb{R}^s $  and $\mathbf{R}=(R_1,...,R_s)\geq 0 $. When we fix all blocks but $x_i$, we write $F(\bar{x}_1,...,x_i,...,\bar{x}_s)$ as $F(x_i,\bar{\vx}_{-i})$. 

\begin{defn}
A point $\bar{\vx}$ is called a blockwise $\mathbf{R}$-local minimizer of $F$ if it satisfies
\[F(\bar{x}_i,\bar{\vx}_{-i}) = \min\limits_{x_i \in B(\bar{x}_i,R_i)}F(x_i,\bar{\vx}_{-i}) ,\quad  1\leq i \leq s,\]
where $F(\bar{\vx})=F(\bar{x}_i,\bar{\vx}_{-i})$.
\end{defn}
When $\mathbf{R}=\infty$, $\bar{\vx}$ is known as a Nash equilibrium point of $F$.


We can obtain similar estimates on the gradient of $f$ for blockwise R-local minimizers. Recall that $S_{\vQ}=\{\vx : \|\nabla f(\vx)\|\leq \delta \}$.
\begin{thm}
Suppose $f$ and $r$ satisfy Assumptions \ref{assumption:f} and \ref{assumption:r}. If $\bar{\vx}$ is a blockwise $\mathbf{R}$-local minimizer of $F$, then $\bar{\vx} \in S_{\vQ}$ (i.e, the property $\vQ$ is met) for $\delta = \|\mathbf{v}\|\coloneqq(\sum|v_i^2|)^{\frac{1}{2}}$ where $v_i:=\alpha+\max\{\frac{4\beta}{R_i},2\sqrt{\beta L} \}$, $1 \leq i \leq s$. 

\label{thm:blockwise nosplit}
\end{thm}
\begin{proof}
$\bar{x}_i$ is an $R_i$-local minimizer of $F(x_i,\bar{\vx}_{-i})$.
Since $F(x_i,\bar{\vx}_{-i})=f(x_i,\bar{\vx}_{-i})+r(x_i,\bar{\vx}_{-i})$ and $f(x_i,\bar{\vx}_{-i}) $ and $ r(x_i,\bar{\vx}_{-i})$ satisfy Assumption \ref{assumption:f} and Assumption \ref{assumption:r}, Theorem \ref{thm:global case no split} shows that
$\|\nabla_i f(\bar{x}_i,\bar{\vx}_{-i})\|\leq \alpha+\max\{\frac{4\beta}{R_i},2\sqrt{\beta L} \}= v_i. $ Hence $\|\nabla f(\bar{\vx})\|\leq \|\mathbf{v}\|$.
\qed

\end{proof}

\begin{rem}
We can also obtain a simplified version of Theorem \ref{thm:blockwise nosplit}, which is $$\|\nabla f(\bar{\vx})\|\leq \delta := \sqrt{s} \left(\alpha+\max\{\frac{4\beta}{\min_i R_i},2\sqrt{\beta L}\} \right).$$ The main difference between the standar and blockwise estimates is the extra factor $\sqrt{s}$ in the latter.
\end{rem}

\begin{rem}
Since we can set $R=\infty$, our results apply to Nash equilibrium points.
\end{rem} 


Generalized from Corollary \ref{cor:global}, the following result provides estimates of $R_i$ for escaping from non-global local minimizers. The estimates are smaller when $\nabla_i f$ are larger.
\begin{corollary}\label{cor:blockwise}
Let $\vx$ be a local minimizer of $F$ and $\|\nabla_i f(x_i,\vx_{-i})\|>\alpha+2\sqrt{\beta L}$ for some $i$. As long as $R_i>\frac{4\beta}{\|\nabla_i f(x_i,\vx_{-i})\|-\alpha}$, there exists $y\in B(x_i,R_i)$, such that $F(y,\vx_{-i})<F(x_i,\vx_{-i})$.
\end{corollary}
The theorem below, which follows from Theorems~\ref{thm:error bound} and~\ref{thm:blockwise nosplit}, bounds the distance of an $\vR$-local minimizer to the set of global minimizers. We do not have a vector $\vR$ to ensure the global optimality of $\bar{\vx}$ due to the blockwise limitation. Of course, after reaching $\bar{\vx}$, if we switch to standard (non-blockwise) inspection to obtain an standard $R$-local minimizer, we will be able to apply Theorem~\ref{thm:global opt global}.
\begin{thm}\label{thm:global opt blockwise}
Suppose $f$ and $r$ satisfy Assumptions \ref{assumption:f}--\ref{assumption:r}. If $\bar{\vx}$ is a blockwise $\mathbf{R}$-local minimizer of $F$, then
\begin{align*}
d(\bar{\vx},\chi^*)\leq \tfrac{2\sqrt{s}}{\mu} \left(\alpha+\max\{\frac{4\beta}{\min_i R_i},2\sqrt{\beta L}\} \right)+M.
\end{align*}
\end{thm}
\subsection{Run-and-Inspect Method}
\label{section:algorithm}
In this section, we introduce our Run-and-Inspect Method. The ``run'' phase can use any algorithm that monotonically converges to an approximate stationary point. 
When the algorithm stops at either an approximate local minimizer or a saddle point, our method starts its ``inspection'' phase, which either moves to a strictly better point or verifies that the current point is an approximate (blockwise) $R$-local minimizer. 


\subsubsection{Approximate $R$-local minimizers}
We define \emph{approximate} $R$-local minimizers. Since an $R$-local minimizer is a special case of a blockwise $\vR$-local minimizer, we only deal with the latter. Let $\vx = [x_1^T\cdots x_s^T]^T$. A point $\bar{\vx}$ is called a blockwise $\vR$-local minimizer of $F$ \emph{up to $\eta=[\eta_1\cdots\eta_s]^T\geq 0$} if it satisfies
\[ F(\bar{x}_i,\bar{\vx}_{-i}) \leq \min\limits_{x_i \in B(\bar{x}_i,R_i)}F(x_i,\bar{\vx}_{-i})+\eta_i, \quad  1\leq i \leq s;\]
when $s=1$, we say $\bar{\vx}$ is an $R$-local minimizer of $F$ up to $\eta$. It is easy to modify the proof of Theorem \ref{thm:global case no split} to get:

\begin{thm}
Suppose $f$ and $r$ satisfy Assumptions \ref{assumption:f} and \ref{assumption:r}. Then $\bar{\vx} \in S_{\vQ}$ if $\bar{\vx}$ is a blockwise $\vR$-local minimizer of $F$ up to $\eta$ and $\delta \geq \|\mathbf{v}\|\coloneqq(\sum|v_i^2|)^{\frac{1}{2}}$ for $v_i=\alpha+\max\{\frac{4\beta+2\eta_i}{R_i},\sqrt{(4\beta+2\eta_i)L}\}$, $1\leq i\leq s$.
\label{thm:approx blockwise R-local}
\end{thm}
Whenever the condition $\bar{\vx} \in S_{\vQ}$ holds, our previous results for blockwise $\vR$-local minimizers are applicable.

\subsubsection{Algorithms}
Now we present two algorithms based on our Run-and-Inspect Method. Suppose that we have implemented an algorithm and it returns a point $\bar{\vx}$. For simplicity let $\mathbf{Alg}$ denote this algorithm. To verify the global optimality of $\bar{\vx}$, we seek to inspect $F$ around $\bar{\vx}$ by sampling some points. Since a global search is apparently too costly, the inspection is limited in a ball centered at $\bar{\vx}$, and for high-dimensional problems, further limited to lower-dimensional balls. 

The inspection strategy is to sample some points in the ball around the current point and stop whenever either a better point is found or it finishes the last point. By ``better", we mean the objective value decreases by at least a constant amount $\nu>0$.
We call this $\nu$ descent threshold. If a better point is found, we resume $\mathbf{Alg}$ at that point. If no better point is found, the current point is an $R$-local or $\vR$-local minimizer of $F$ up to $\eta$, where $\eta$ depends on the density of sample points and the Lipschitz constant of $F$ in the ball. 
\begin{algorithm}[H]
\caption{Run-and-Inspect Method} 
\begin{algorithmic}[1]
\STATE Set $k=0$ and choose $\vx^0 \in \mathbb{R}^n$;\\
\STATE Choose the descent threshold $\nu > 0$;
\LOOP
    \STATE $\bar{\vx}^k=\mathbf{Alg}(\vx^k)$;
    \STATE Generate a set $\cS$ of sample points in $B(\bar{\vx}^k,R)$;
    \IF {there exists $\vy\in\cS$ such that $F(\vy) < F(\bar{\vx}^k)-\nu$}
        \STATE $\vx^{k+1}=\vy$;
        \STATE $k=k+1$;
    \ELSE
        \STATE \textbf{stop} and \textbf{return} $\bar{\vx}^k$;
    \ENDIF
\ENDLOOP
\end{algorithmic}
\label{algorithm:inspecting process}
\end{algorithm}
If $\mathbf{Alg}$ is a descent method, i.e., $F(\bar{\vx}^k)\leq F(\vx^k)$, algorithm \ref{algorithm:inspecting process} will stop and output a point $\bar{\vx}^{k^*}$ within finitely many iterations: 
$k^*\leq\frac{F(\vx_0)-F^*}{\nu},$
where $F^*$ is the global minimum of $F$.

The sampling step is a \emph{hit-and-run}, that is, points are only sampled when they are used, and the sampling stops whenever a better point is obtained (or all points have been used). The method of sampling and the number of sample points can vary throughout iterations and depend on the problem structure.
In general, sampling points from the outside toward the inside is more efficient. 

Here, we analyze a simple approach in which sufficiently many well-distributed points are sampled to ensure that $\bar{\vx}^{k^*}$ is an approximate $R$-local minimizer. 



\begin{thm}
Assume that $f(\vx)$ is $\bar{L}$-Lipschitz continuous\footnote{Do not confuse with the Lipschitz constant $L$ of $\nabla f$.} in the ball $B(\bar{\vx},R)$ and the set of sample points $\cS=\{\vy_1,\vy_2,...,\vy_m\}$ has density $\bar{r}$, that is,
\[\max_{\vx \in B(\bar{\vx},R)}\min_{0 \leq j \leq m}\|\vx-\vy_{j}\| \leq \bar{r} \]
where $\vy_{0}=\bar{\vx}$. If
\begin{align*}
F(\vy_j)\geq F(\bar{\vx})-\nu,\quad j=1,2,\ldots,m,
\end{align*}
then the point $\bar{\vx}$ is an $R$-local minimizer of $F$ up to $\eta=\nu + (\bar{L}+\alpha)\bar{r}+2\beta$.
\label{thm:approximate R-local}
\end{thm}
\begin{proof}
Suppose $\min\limits_{\vx \in B(\bar{\vx},R)} F(\vx)$ is attained at $\tilde{\vx}$.
Since $\cS$ has density $r$, we have $\vy_j$ such that
\begin{align*}
  \|f(\tilde{\vx})-f(\vy_{j})\| & \leq \bar{L} \bar{r},\\  
  \|r(\tilde{\vx})-r(\vy_{j})\| & \leq \alpha \bar{r}+ 2\beta.
\end{align*}
From $F(\bar{\vx})- \nu \le F(\vy_{j}) $ and $ F(\vy_{j})= F(\tilde{\vx}) + (F(\vy_{j})-F(\tilde{\vx}))$, it follows
\begin{align*}
F(\bar{\vx}) &\leq \min\limits_{\vx \in B(\bar{\vx},R)} F(\vx)+ \nu + (\bar{L}+\alpha)\bar{r}+2\beta.
\tag*{\qed}
\end{align*}
\end{proof}
When the dimension of $\vx$ is high , it is impractical to inspect over a high-dimensional ball. This motivates us to extend algorithm \ref{algorithm:inspecting process} to its blockwise version.  
\begin{algorithm}[H]
\caption{Run-and-Inspect Method (blockwise version)} 
\begin{algorithmic}[1]
\STATE Set $k=0$ and choose $\vx^0 \in \mathbb{R}^n$;\\
\STATE Choose the descent threshold $\nu > 0$;
\LOOP
    \STATE $\bar{\vx}^k=\mathbf{Alg}(\vx^k)$;
    \STATE Generate sets $\cS_i$ of sample points in $B(\bar{x}^k_i,R_i)$ for $i=1,...,s$;
    \IF {there exist $i$ and $z\in \cS_i$ such that $F(z,\bar{\vx}^k_{-i}) < F(\bar{x}^k,\bar{\vx}^k_{-i})- \nu$}
        \STATE $x_i^{k+1}= z$;
        \STATE $x_j^{k+1}=\bar{x}_j^k$ for all $j\neq i$;
        \STATE $k=k+1$;
    \ELSE
        \STATE \textbf{stop} and \textbf{return} $\bar{\vx}^k$;
    \ENDIF
\ENDLOOP
\end{algorithmic}
\label{algorithm:blockwise inspecting process}
\end{algorithm}

Algorithm \ref{algorithm:blockwise inspecting process} samples points in a block while keeping other block variables fixed. This algorithm ends with an approximate blockwise $\vR$-local minimizer. 

\begin{thm}
Assume that $f(x_i,\bar{\vx}_{-i})$ is $\bar{L}_i$-Lipschitz continuous in the ball $B(\bar{x}_i,R_i)$ for $1\leq i\leq s$ and the set of sample points $\cS_i=\{z_{i1},z_{i2},\ldots,z_{im_i}\}, 1\leq i\leq s$ has blockwise-density $\bar{r}$, that is,
\[\max_{x_i \in B(\bar{x}_i,R_i)}\min_{0 \leq j \leq m_i}\|x_i-z_{ij}\| \leq \bar{r} , \quad \forall 1 \leq i \leq s, \]
where $z_{i0}=\bar{x}_i$.
If
\begin{align*}
F(z_{ij},\bar{\vx}_{-i})\geq F(\bar{\vx})-\nu,\quad j=1,2,\ldots,m_i, i=1,2\ldots,s,
\end{align*}
then $\bar{\vx}$ is a blockwise $\vR$-local minimizer of $F$ up to $\eta=[\eta_1,\dots,\eta_s]^T$ for $\eta_i=\nu+(\bar{L}_i+\alpha)\bar{r}+2\beta$.
\label{thm:approximate blockwise R-local}
\end{thm}
The proof is similar to that of Theorem \ref{thm:approximate R-local}.

The next proposition states that inspection around a point with sufficiently large partial gradient of $f$ ensures a sufficient descent.
\begin{prop}
\label{thm:density}
Assume that we sample points in the ball $B(\bar{x}_i,R_i)$ with density $\bar{r}\leq\frac{R_i}{2}$. If $\|\nabla_i f(\bar{x}_i,\bar{\vx}_{-i})\|\geq\frac{9}{2}L_i\bar{r}+3\alpha+\frac{2\beta+\nu}{\bar{r}}$, then there exists at least one sampled
point $z_i$ which satisfies
\begin{align}
    F(z_i,\bar{\vx}_{-i})<F(\bar{\vx})-\nu.
\end{align}
\end{prop}
\begin{proof}
Let $x_i = \bar{x}_i-2\frac{\nabla_i f(\bar{\vx}_i)}{\|\nabla_i f(\bar{\vx}_i)\|}\bar{r}$. There exists $z_i\in B(\bar{x}_i,R_i)$ such that $\|x_i-z_i\|\leq \bar{r}$. Then
\begin{align*}
F(z_i,\bar{\vx}_{-i})&=f(z_i,\bar{\vx}_{-i})+r(z_i,\bar{\vx}_{-i})\leq f(z_i,\bar{\vx}_{-i})+r(\bar{\vx})+3\alpha\bar{r}+2\beta\\
&\leq f(\bar{\vx})+\dotp{\nabla_i f(\bar{x}_i,\bar{\vx}_{-i}),z_i-\bar{x}_i}+\frac{L_i}{2}\|z_i-x_i\|^2+r(\bar{\vx})+3\alpha\bar{r}+2\beta\\
&=F(\bar{\vx})+\dotp{\nabla_i f(\bar{x}_i,\bar{\vx}_{-i}),z_i-x_i}+\dotp{\nabla_i f(\bar{x}_i,\bar{\vx}_{-i}),x_i-\bar{x}_i}+\frac{L_i}{2}\|z_i-x_i\|^2+3\alpha\bar{r}+2\beta\\
&\leq F(\bar{\vx})-\|\nabla_i f(\bar{x}_i,\bar{\vx}_{-i})\|\bar{r}+\frac{9}{2}L_i\bar{r}^2+3\alpha\bar{r}+2\beta\\
&\leq F(\bar{\vx})-\nu.
\tag*{\qed}
\end{align*}
\end{proof}
Therefore $\bar{L}_i$ in Theorem \ref{thm:approximate blockwise R-local} can be bounded by $\frac{9}{2}L_i\bar{r}+3\alpha+\frac{2\beta+\nu}{\bar{r}}+L_iR_i$. And we can set
\begin{align*}
   \eta_i = 2\nu+\frac{9}{2}L_i\bar{r}^2+4\alpha\bar{r}+4\beta+L_iR_i\bar{r}.
\end{align*}
This bound is not tight when $R_i$ is very large. 

\subsection{Complexity analysis}
Since Algorithm \ref{algorithm:blockwise inspecting process} generalizes Algorithm \ref{algorithm:inspecting process} to multiple blocks, we analyze the complexity of the former.

There are quite many parameters that affect the complexity results. In this analysis, we focus on the dimension of the space $d$ and different options to set the dimension $d'$ of each block, assuming that all blocks have the same dimension $d'\le d$ and $d'$ evenly divides $d$, thus, creating exactly $s=d/d'$ blocks of variables. 
We assume that the smoothness and strong-convexity parameters of $f$ are $L_i=L\geq\mu$, respectively. Of course, $L,\mu$ affect the complexity significantly though not as much as the dimensions (unless $L,\mu$ themselves depend on $d$). Assume the function $r$ satisfies Assumption \ref{assumption:r} with parameters {$\alpha,\beta\in [0,1)$}. 
The $\vR$-local min tolerance $\eta_i$ of each block $i$ is tied to $\beta$, $\nu$, $\bar{r}$ (density of sample points), and $R_i$. Based on Theorem \ref{thm:approximate blockwise R-local} and Proposition \ref{thm:density} and using free parameters $c_{\nu}, c_{\bar{r}}, t$ that will be tuned later, we set $\nu:=c_{\nu}\beta$, $\bar{r}:=\sqrt{\frac{c_{\bar{r}}\beta}{L}}$, $R_i\equiv R:=t\bar{r}$ for some $t\geq 2$ and thus get $\eta_i\equiv\eta=2\nu+\frac{9}{2}L_i\bar{r}^2+4\alpha\bar{r}+4\beta+L_iR_i\bar{r}= \left(2c_{\nu}+\left(\frac{9}{2}+t\right)c_{\bar{r}}+4\right)\beta+4\alpha\sqrt{\frac{c_{\bar{r}}\beta}{L}}$.


Remember Theorem \ref{thm:approx blockwise R-local} gives a bound on $\delta$ for Algorithm \ref{algorithm:blockwise inspecting process}, 
\begin{align}
    \delta &\geq \sqrt{\frac{d}{d'}}\left(\alpha+\max\{\frac{4\beta+2\eta}{R},\sqrt{(4\beta+2\eta)L}\}\right)=\sqrt{\frac{d}{d'}}\left(\alpha+c_1\sqrt{\beta L}\right),
\label{eqn:delta bound}
\end{align}
where $c_1=\max\{\frac{4+2\frac{\eta}{\beta}}{t\sqrt{c_{\bar{r}}}},\sqrt{4+2\frac{\eta}{\beta}}\}$. Using this $\delta$, Theorem \ref{thm:error bound} produces global error bounds. We can calculate that our Algorithm \ref{algorithm:blockwise inspecting process} using $R_i=t\bar{r}$ will return an approximate $\vR$-local minimizer $\bar{\vx}$ satisfying the global error bound:
\begin{align}
    F(\bar{\vx})-F(\vx^*) \leq \frac{\delta^2+2\alpha \delta}{\mu}+ 2\beta= c_2\left(\frac{d}{d'}\frac{\left(\alpha+c_1\sqrt{\beta L}\right)^2}{\mu}\right),
    \label{eqn:rough bound1}
\end{align}
where $c_2=1+2\frac{\alpha}{\delta}+\frac{2\mu\beta}{\delta^2}\leq 4$.

For simplicity, we set $c_{\bar{r}}=c_{\nu}=1$, $t=6$ and assume $\alpha<\sqrt{\beta L}$. By \eqref{eqn:delta bound}, \eqref{eqn:rough bound1}, we will get
\begin{align}
\label{eqn:rough bound}
    F(\bar{\vx})-F(\vx^*)\leq 4\left(\frac{d}{d'}\frac{\left(\alpha+7.5\sqrt{\beta L}\right)^2}{\mu}\right)
\end{align}
Next, we take three steps to calculate the complexity of Algorithm \ref{algorithm:blockwise inspecting process}:
\begin{itemize}
    \item Since \textbf{Alg} is a descent algorithm and each inspection decreases the objective error by at least $\nu=c_{\nu}\beta$, with initial point $x^0$, we need at most $O\left(\frac{F(\vx^0)-F(\vx^*)}{c_{\nu}\beta}\right)$ inspections or loops in Algorithm \ref{algorithm:blockwise inspecting process}. Under our assumption $c_{\nu}=1$, the number of inspections is $O\left(\frac{F(\vx^0)-F(\vx^*)}{\beta}\right)$
    \item In each loop, Algorithm \ref{algorithm:blockwise inspecting process} runs \textbf{Alg} with a complexity $C_{\textbf{Alg}}$ and perform an inspection.
    \item Each inspection involves at most $\frac{d}{d'}$ blocks. When inspecting each $d'$-dimensional block, we sample an $\bar{r}$-net. For simplicity, we just sample the points on a uniform grid. Now the space are partitioned into many squares. The distance between the center and grid points of each square should be the density $\bar{r}$, for which the grid width needs to be $\frac{2\bar{r}}{\sqrt{d'}}$. Hence, the volumn of each square is $v_{d'}(\bar{r})=\left(\frac{2\bar{r}}{\sqrt{d'}}\right)^{d'}$. The volumn of the inspection ball is $V_{d'}(R_i)=\frac{\pi^{\frac{d'}{2}}R_i^{d'}}{\Gamma(\frac{d'}{2}+1)}$. By Stirling's formula, $\Gamma(\frac{d'}{2}+1)\approx \sqrt{2\pi}(\frac{d'}{2})^{\frac{d'}{2}+\frac{1}{2}}e^{-\frac{d'}{2}}$. Then the total number of the inspected points is around
    \begin{align*}
        \frac{V_{d'}(R_i)}{v_{d'}(\bar{r})}\approx\frac{\pi^{\frac{d'}{2}}R_i^{d'}}{\sqrt{2\pi}(\frac{d'}{2})^{\frac{d'}{2}+\frac{1}{2}}e^{-\frac{d'}{2}}\left(\frac{2\bar{r}}{\sqrt{d'}}\right)^{d'}}=\left(\frac{\pi e}{2}\right)^{\frac{d'}{2}}\frac{1}{\sqrt{\pi d'}}\left(\frac{R_i}{\bar{r}}\right)^{d'}=e^{\Theta\left(d'\left(\log\left(\frac{R_i}{\bar{r}}\right)+\frac{1}{2}\log\left(\frac{\pi e}{2}\right)\right)\right)}.
    \end{align*}
Using our choice $R=t\bar{r}$ and $\frac{1}{2}\log\left(\frac{\pi e}{2}\right)\approx 0.54$, we get
    \begin{align*}
        \frac{V_{d'}(R_i)}{v_{d'}(\bar{r})}=e^{\Theta\left(d'\left(\log(t)+0.54\right)\right)},
    \end{align*}
where $t=6$ under our assumption. 

\end{itemize}


The complexity of Algorithm \ref{algorithm:blockwise inspecting process} is the product of the number of loops and the complexity of each loop:
\begin{align}
\label{eqn:simplified complexity}
    O\left(\frac{F(\vx^0)-F(\vx^*)}{\beta}\frac{d}{d'}e^{O\left(d'\log\left(6\right)+0.54 d'\right)}\right).
\end{align}

\begin{itemize}
    \item If we choose $d'= \Theta(d)$, then $\frac{d}{d'}=\Theta(1)$. In this case by \eqref{eqn:rough bound}, our algorithm can reach a good accuracy of $ O\left(\frac{\left(\alpha+7.5\sqrt{\beta L}\right)^2}{\mu}\right)$. On the other hand, the inspection complexity \eqref{eqn:simplified complexity} is exponential in $d$.
    \item If we go to the other extreme by choosing $d'=\Theta(1)$, then $\frac{d}{d'}=\Theta(d)$. The complexity \eqref{eqn:simplified complexity} reduce to a polynomial in $d$, but the accuracy becomes worse, $O\left(d\frac{\left(\alpha+7.5\sqrt{\beta L}\right)^2}{\mu}\right)$. In general, it becomes proportional in the dimension $d$. Except, if the function $r$ is very nice with  $\alpha=O(\frac{1}{\sqrt{d}}),\beta\ =O(\frac{1}{d})$, then the relative accuracy is still good at $O\left(1\right)$.
    \item In general, we can choose $d'=\Theta(d^v)$ for $v\in(0,1)$, where the choice of $v$ controls the tradeoff between the accuracy and complexity.
\end{itemize}

\section{Numerical experiments}
\label{section:numerical experiments}
In this section, we apply our Run-and-Inspect Method to a set of nonconvex problems. 
We admit that it is difficult to apply our theoretical results because the implicit decomposition $F=f+r$ with $f,r$ satisfying their assumptions is not known.
Nonetheless, 
The encouraging experimental results below demonstrate the effectiveness of our Run-and-Inspect Method on nonconvex problems even though they may not have the decomposition.


\subsection{Test example : Ackley's function}
The Ackley function is widely used for testing optimization algorithms, and in $\mathbb{R}^2$, it has the form
\[f(x,y)=-20e^{-0.2\sqrt{0.5(x^2+y^2)}}-e^{0.5(\cos 2\pi x+\cos 2\pi y)}+e+20.\]
   \begin{figure}[H]
        \centering
        \includegraphics[width=150pt]{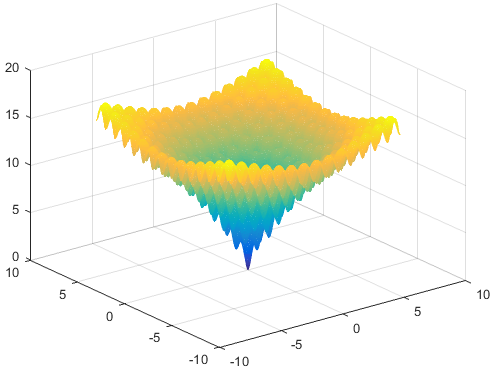}
        \caption{Landscape of Ackley's function in $\mathbb{R}^2$.}
    \end{figure}
The function is symmetric, and its oscillation is regular. 
To make it less peculiar, we 
modify it to an asymmetric function:
\begin{align}\label{eq:F1}
F(x,y)=-20e^{-0.04(x^2+y^2)}-e^{0.7(\sin(xy)+\sin y)+0.2\sin(x^2)}+20.
\end{align}
    \begin{figure}[ht]
        \centering
        \includegraphics[width=6cm]{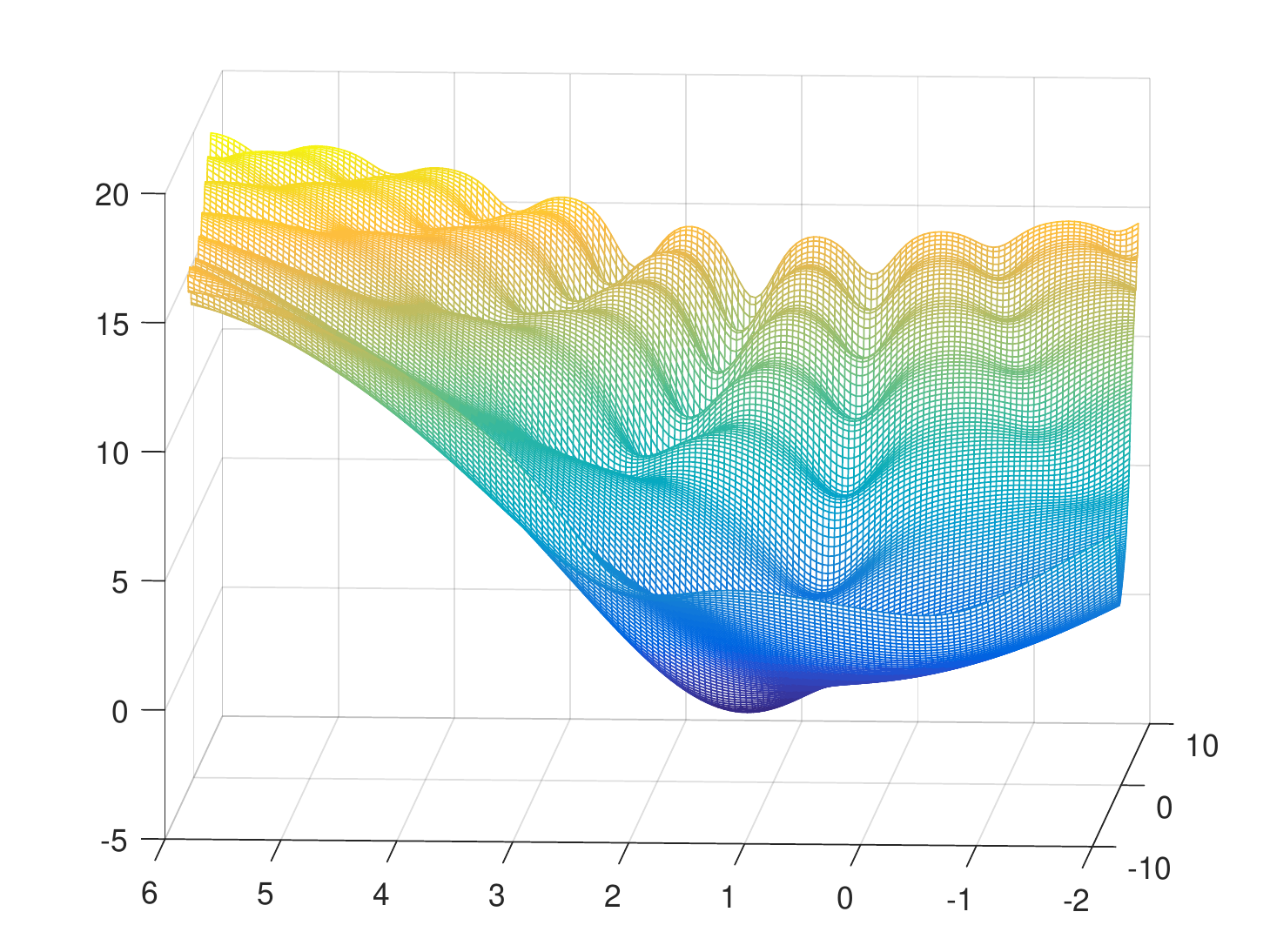}\includegraphics[width=6cm]{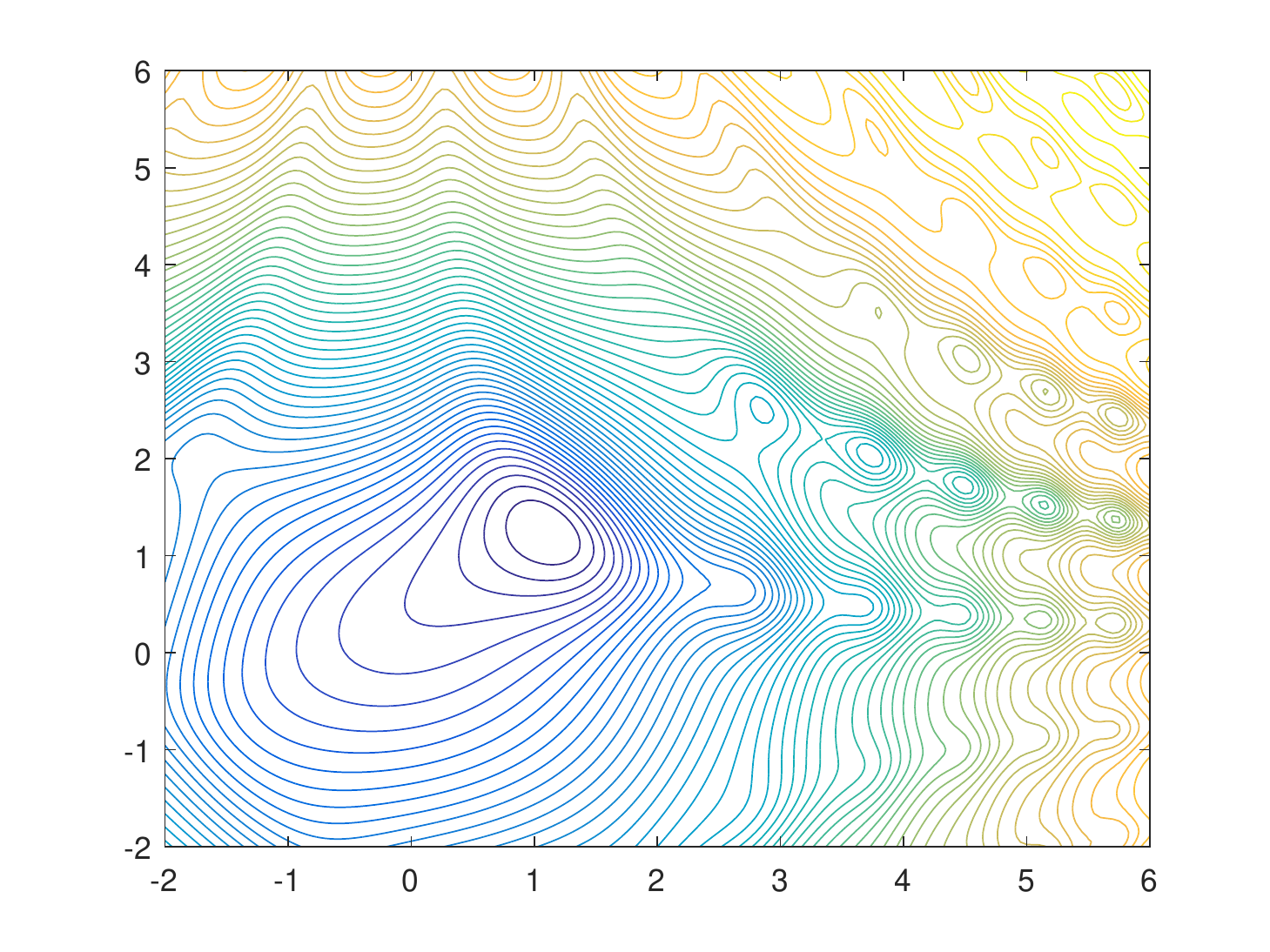}
        \caption{Landscape and contour of $F$ in \eqref{eq:F1}.}
    \end{figure}
    
The function $F$ in \eqref{eq:F1} has a lot of local minimizers, which are irregularly distributed. If we simply use the gradient descent (GD) method without a good initial guess, it will converge to a nearby local minimizer. To escape from local minimizers, we conduct our Run-and-Inspect Method according to Algorithms~\ref{algorithm:inspecting process} and \ref{algorithm:blockwise inspecting process}. We sample points starting from the outer of the ball toward the inner. The radius $R$ is set as $1$ and $\Delta R$ as $0.2$. $\mathbf{Alg}$ is GD and block-coordinate descent (BCD), and we apply two-dimensional inspection and blockwise one-dimensional inspection to them, respectively. The step size of GD and BCD is $1/40$. The results are shown in Figures \ref{test example:2D inspect} and \ref{test example:blockwise inspect}, respectively. Note that the ``run'' and ``inspect'' phases can be decoupled, so a blockwise inspection can be used with either standard descent or blockwise descent algorithms. 
    \begin{figure}[ht]
        \centering
        
        \includegraphics[width=6cm]{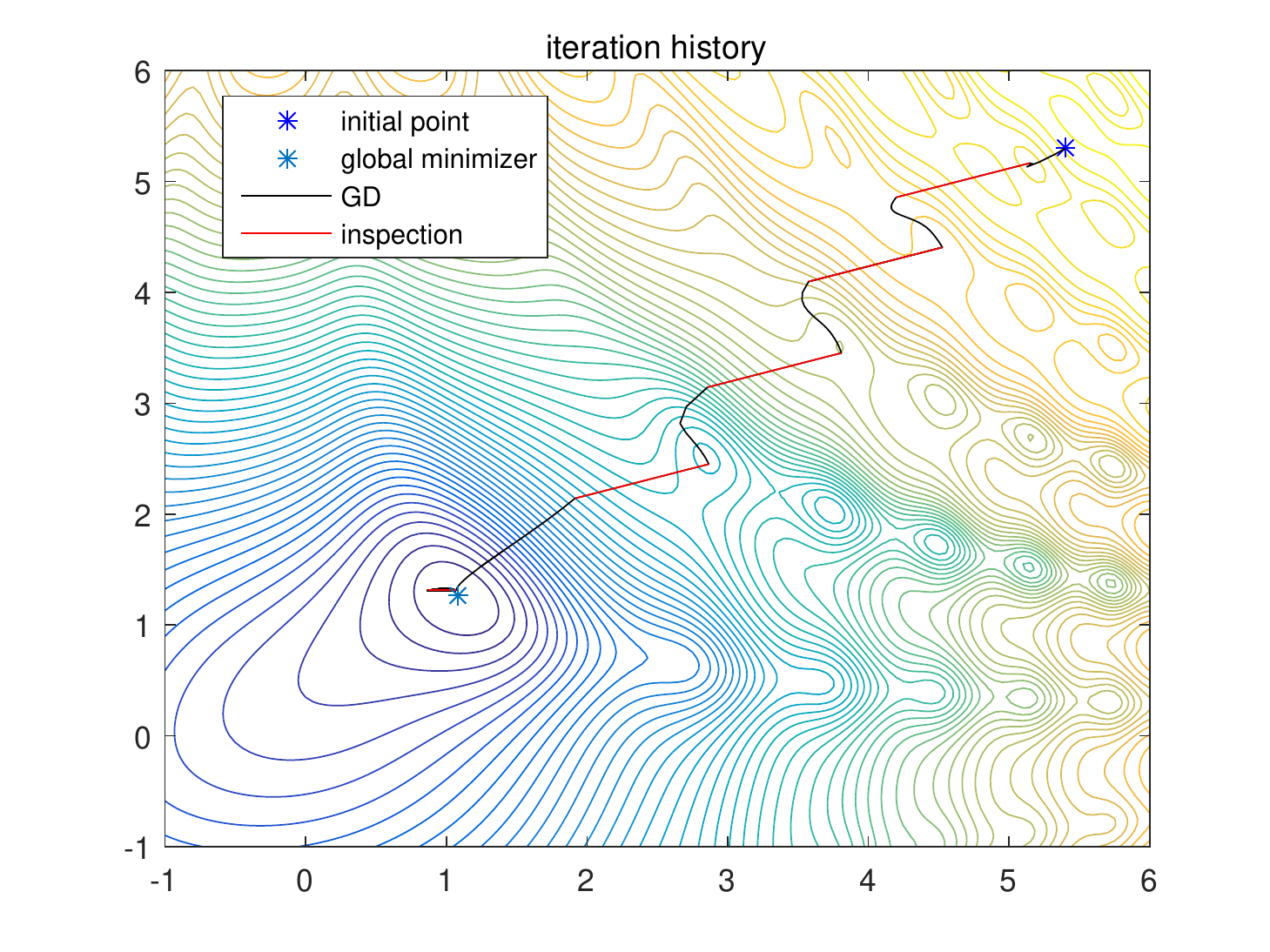}\includegraphics[width=6cm]{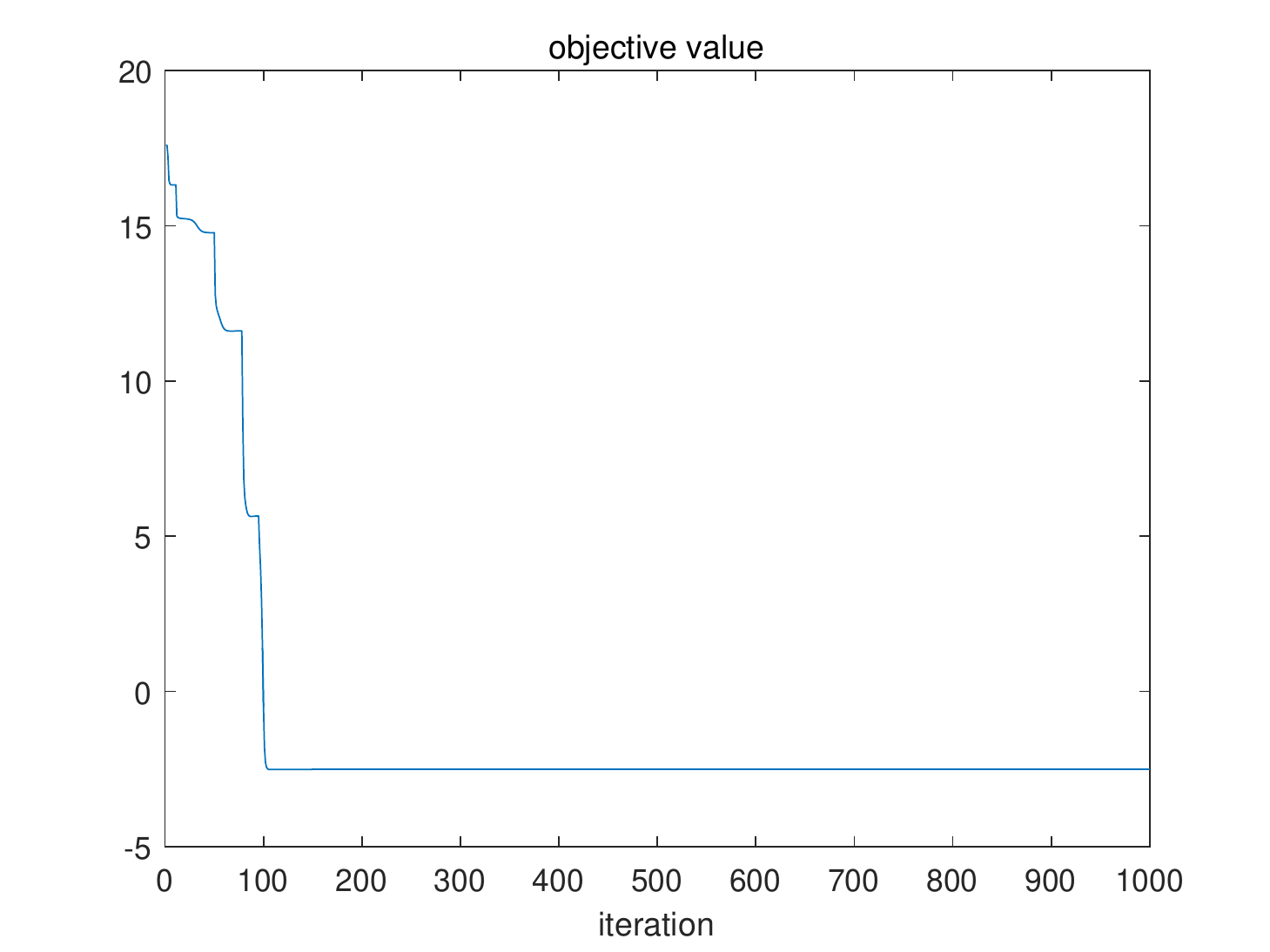}
        \caption{GD iteration with 2D inspection}
        \label{test example:2D inspect}
    \end{figure}
    
        \begin{figure}[ht]
        \centering
        \includegraphics[width=6cm]{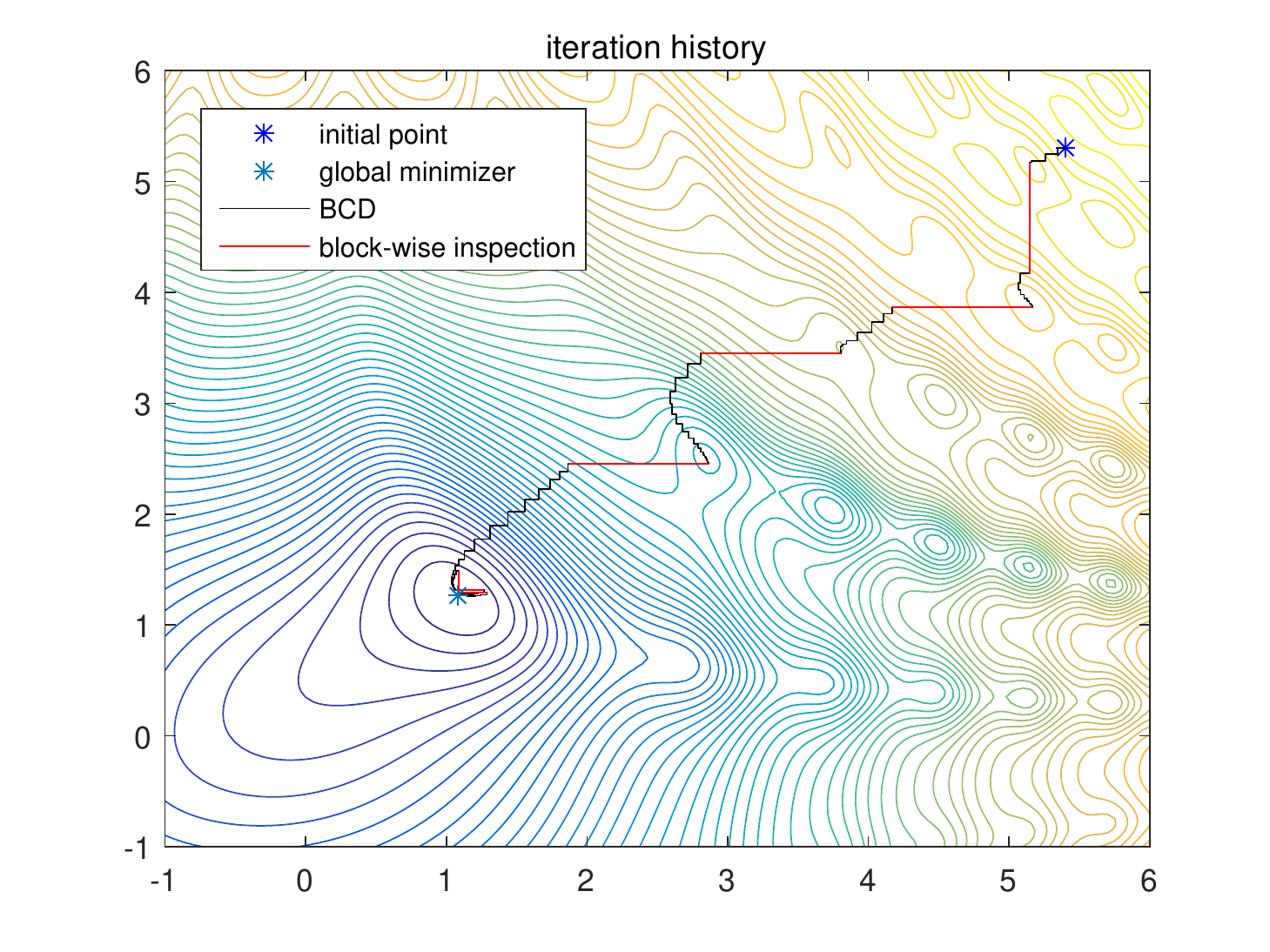}\includegraphics[width=6cm]{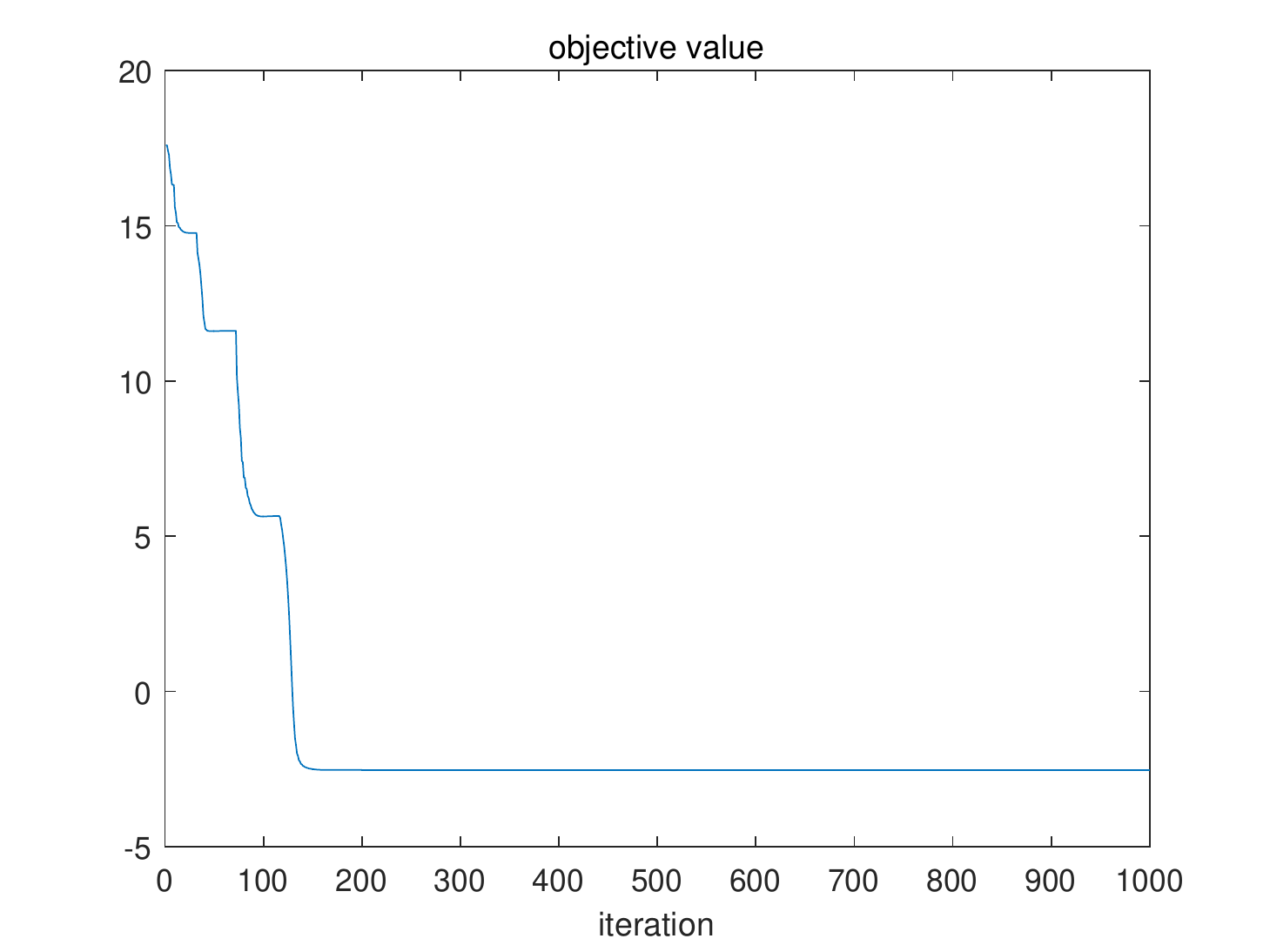}
        \caption{BCD iteration with blockwise 1D inspection}
        \label{test example:blockwise inspect}
    \end{figure}
    
    From the figures, we can observe that blockwise inspection, which is much cheaper than standard inspection, is good at jumping out the valleys of local minimizers. Also, the inspection usually succeeds very quickly at the large initial value of $R$, so it is swift. These observations guide our design of inspection. 
    Although smaller values of $R$ are sufficient to escape from local minimizers, especially those that are far away from the global minimizer, we empirically use a rather large $R$ and, to limit the number of sampled points, a relatively large $\Delta R$ as well.
    
    When an iterate is already (near) a global minimizer, there is no better point for inspection to find, so the final inspection will go through all sample points in $B(\bar{\vx},R)$, taking very long to complete, unlike the rapid early inspections. In most applications, however, this seems unnecessary. If $F$ is smooth and strongly convex near the global minimizer $\vx^*$, we can theoretically eliminate spurious local minimizers in $B(\bar{\vx},R')$ and thus search only in the smaller region $B(\bar{\vx},R)\setminus B(\bar{\vx},R')$. Because the function $r$ can be nonsmooth in our assumption, we do not have $R'>0$. But, our future work will explore more types of $r$. It is also worth mentioning that, in some applications, global minimizers can be recognized, for example, based on they having the desired structures, achieving the minimal objective values, or attaining certain lower bounds. If so, the final inspection can be completely avoided.

\subsection{K-means clustering}
Consider applying $k$-means clustering to a set of data $\{x_i\}_{i=1}^n\subset \RR^d$. We assume there are $K$ clusters $\{z_i\}_{i=1}^K$ and have the variables $\vz=[z_1,...z_K]\in \RR^{d\times K}$. The problem to solve is
\[\min_{\vz\in \RR^{d\times K}} f(z)= \frac{1}{2n}\sum_{i=1}^n \min_{1\leq j \leq K} \|x_i-z_j\|^2.\]
A classical algorithm is the Expectation Minimization (EM) method, but it is susceptible to local minimizers. We add inspections to EM to improve its results.

We test the problems in~\cite{yin2017stochastic}. The first problem has synthetic Gaussian data in $\mathbb{R}^2$. A total of 4000 synthetic data points are generated according to four multivariate Gaussian distributions with 1000 points on each, so there are four clusters. Their means and covariance
matrices are:
\begin{gather*}
\mu_1=\begin{bmatrix} -5 \\ -3 \end{bmatrix},~
\mu_2=\begin{bmatrix} 5 \\ -3 \end{bmatrix},~
\mu_3=\begin{bmatrix} 0 \\ 5 \end{bmatrix},~
\mu_4=\begin{bmatrix} 2.5 \\ 4 \end{bmatrix}; \\
\Sigma_1=\begin{bmatrix} 0.8 & 0.1\\ 0.1 & 0.8 \end{bmatrix},~
\Sigma_2=\begin{bmatrix} 1.2 & 0.6\\ 0.6 & 0.7 \end{bmatrix},~
\Sigma_3=\begin{bmatrix} 0.5 & 0.05\\ 0.05 & 1.6 \end{bmatrix},~
\Sigma_4=\begin{bmatrix} 1.5 & 0.05\\ 0.05 & 0.6 \end{bmatrix}.
\end{gather*}

The EM algorithm is an iteration that alternates between labeling each data point (by associating it to the nearest cluster center) and adjusting the locations of the centers. 
When the labels stop updating, we start an inspection. In the above problem, the dimension of $z_i$ is two, and we apply a 2D inspection on $z_i$ one after one with radius $R=10$, step size $\Delta R=2$, and angle step size $\Delta\theta=\pi/10$. The descent threshold is $\nu=0.1$.

The results are presented in Figure \ref{fig: gaussian}. We can see that the EM algorithm stops at a local minimizer but, with the help of inspection, it escapes from the local minimizer and reaches the global minimizer. This escape occurs at the first sample point in the  $3$rd block at radius $10$ and angle $\theta=7\pi/10$. Since the inspection succeeds on the perimeter of the search ball, it is rapid.

        
    
\begin{figure}[ht]
    \centering
     \includegraphics[width=6cm]{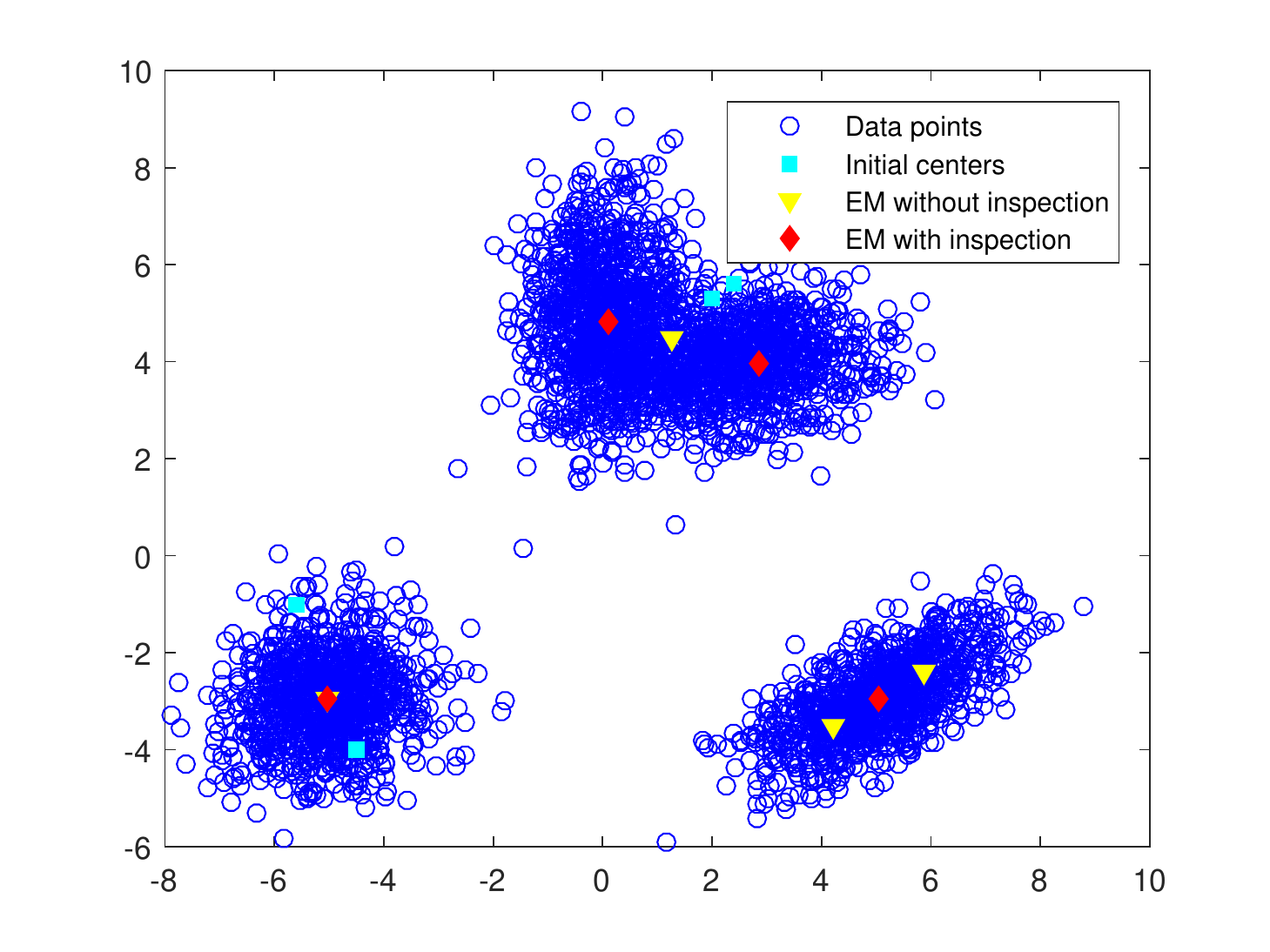}\includegraphics[width=6cm]{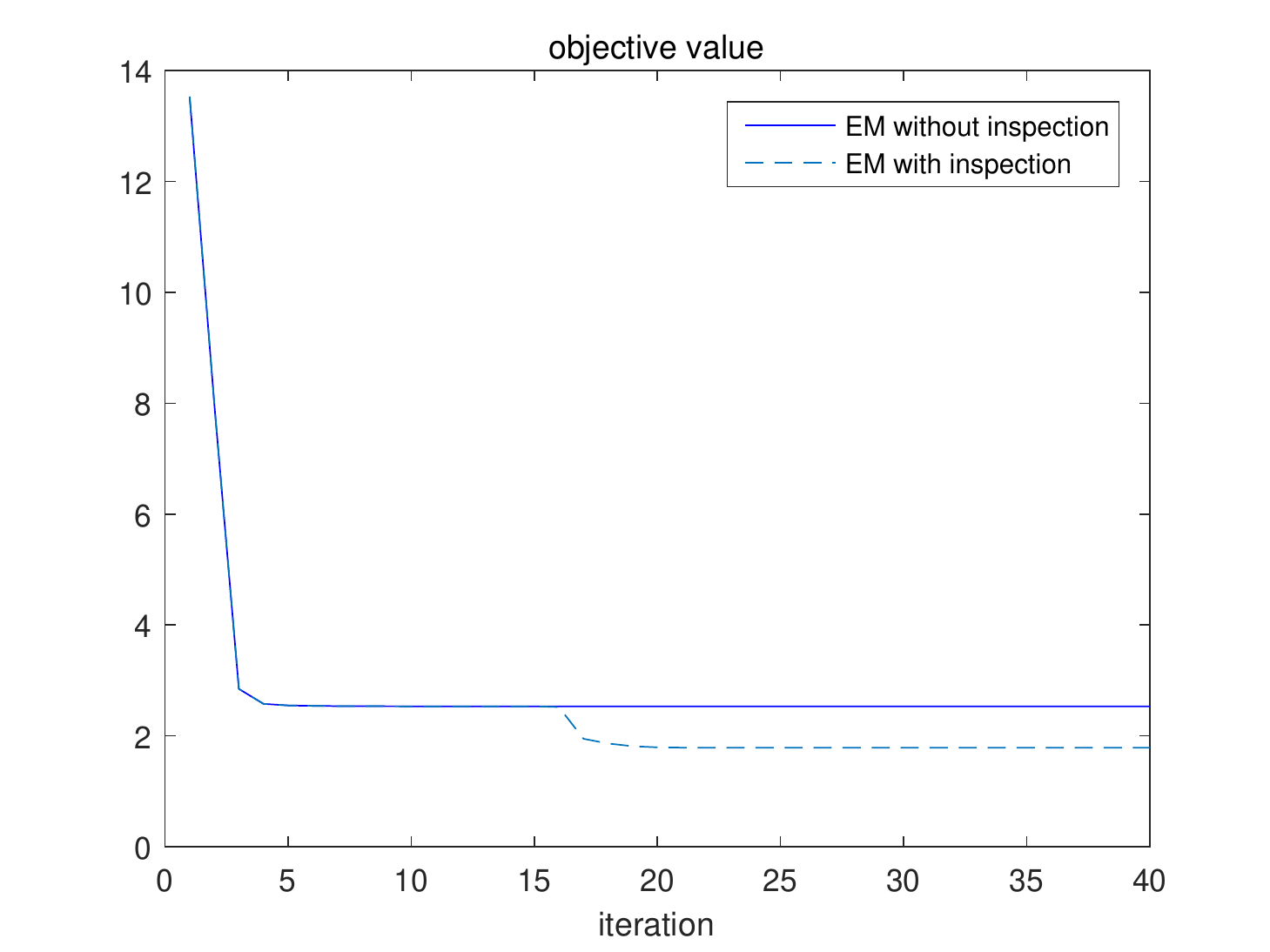} 
 \caption{Synthetic Gaussian data with 4 clusters. \protect Left: clustering result; Right: objective value in the iteration}
 \label{fig: gaussian}
\end{figure}

We also consider the Iris dataset\footnote{\url{https://archive.ics.uci.edu/ml/datasets/iris}}, which contains 150 4-D data samples from 3 clusters. We compare the performance of the EM algorithm with and without inspection over 500 runs with their initial centers randomly selected from the data samples. We inspect the 4-D variables one after one. Rather than sampling the 4-D polar coordinates, which needs three angular axes, we only inspect two dimensional balls. That is, for center $i_0$ and radius $r$, the inspections sample the following points $z_{i_0}$ that has only two angular variables $\theta_1,\theta_2$:
\[z^{{\rm inspected}}_{i_0}=z_{i_0}+r\times [\cos\theta_1~\sin\theta_1 ~ \cos\theta_2 ~\sin \theta_2]^T.
\]
Such inspections are very cheap yet still effective. Similar lower-dimensional inspections should be used with high dimensional problems. We choose $R=3$, $\Delta R=1$, $\Delta\theta_1=\Delta\theta_2=\pi/10$, and a descent threshold $\nu=10^{-3}$. The results are shown in Figures \ref{fig: iris_histogram} and \ref{fig: iris_result}.
\begin{figure}[ht]
    \centering
 \includegraphics[width=8cm]{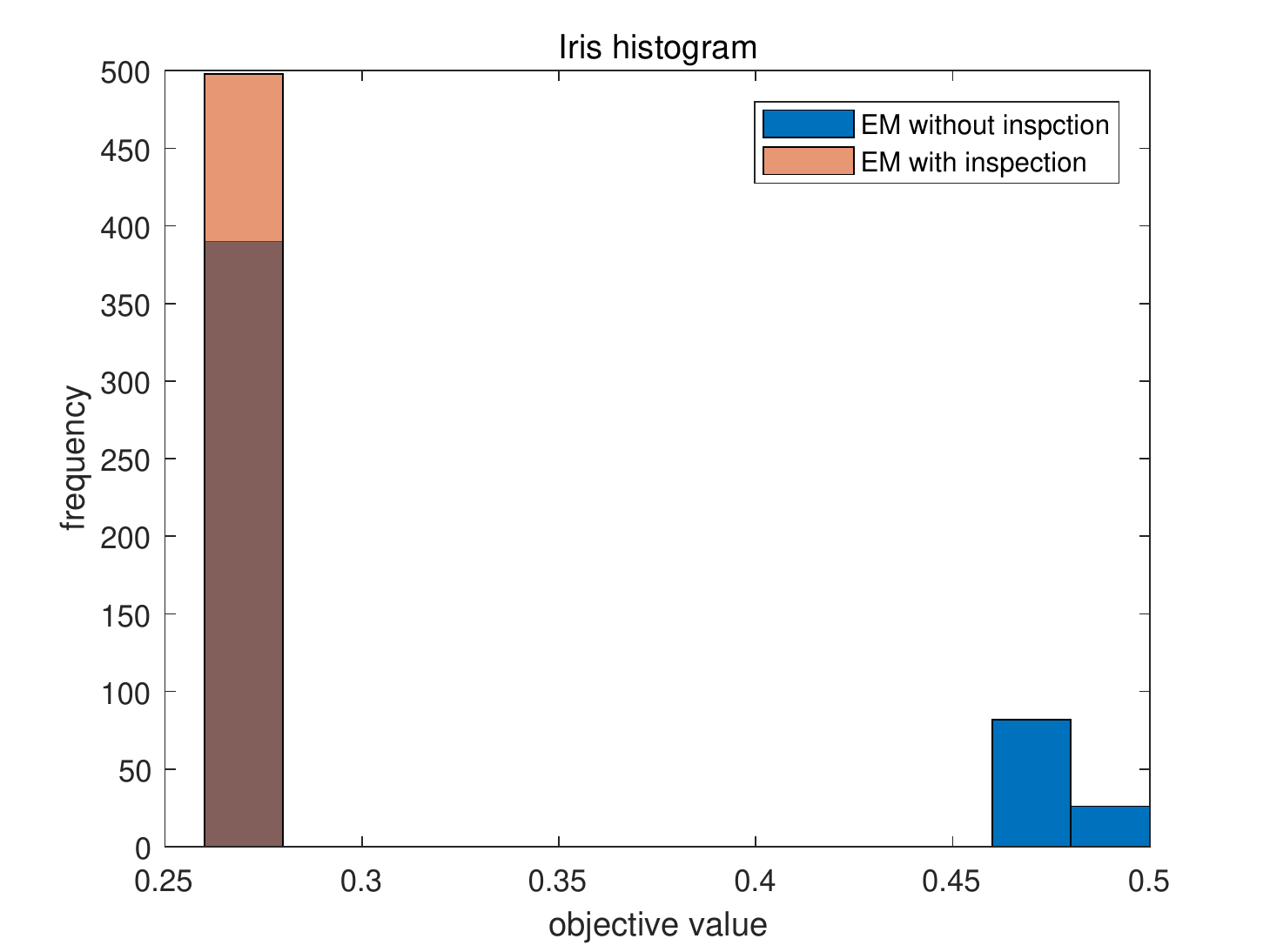}
 \caption{histogram of the final objective values in the 500 experiments}
 \label{fig: iris_histogram}
\end{figure}
\begin{figure}[ht]
    \centering
     \includegraphics[width=6cm]{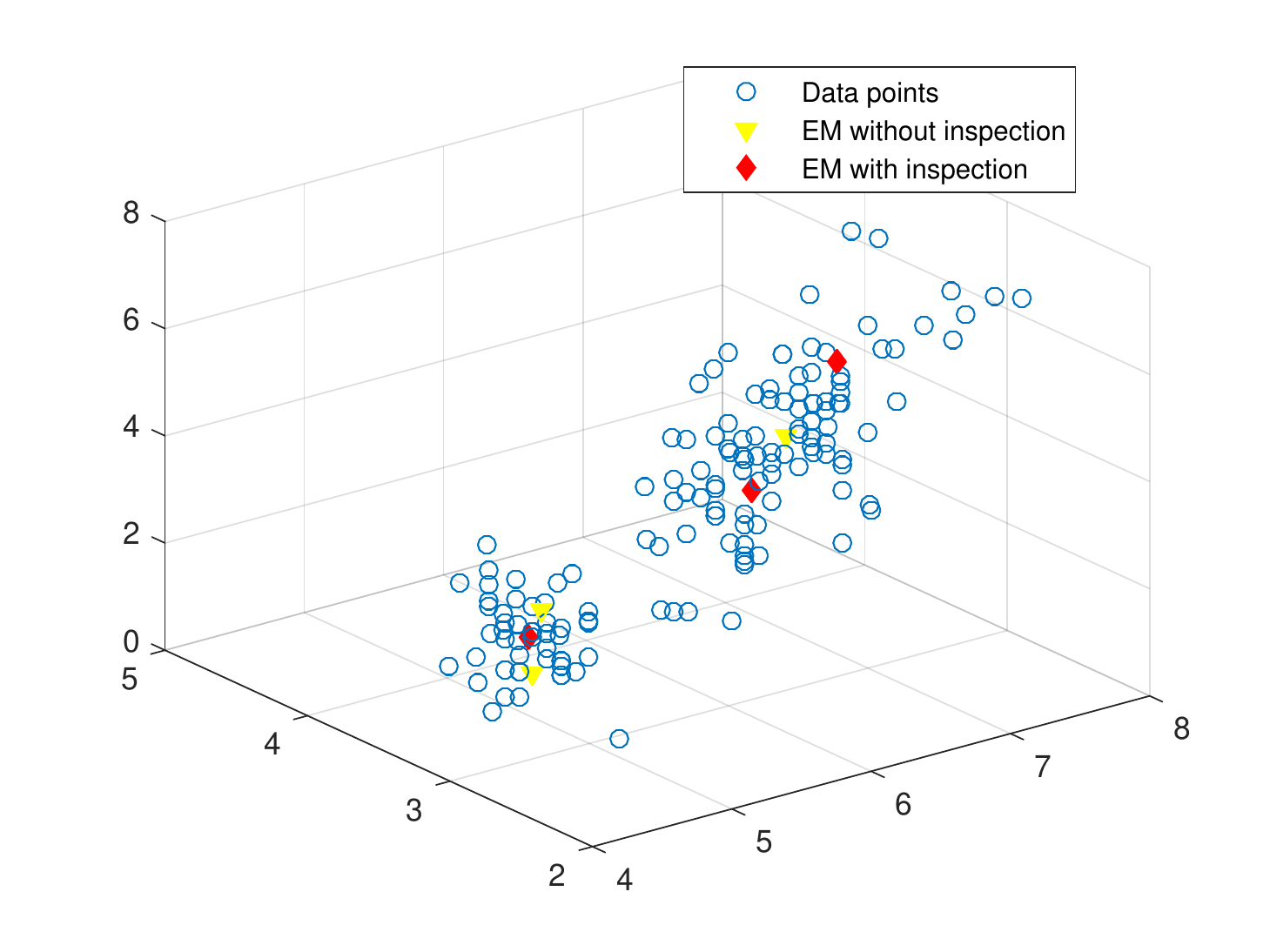}
     \includegraphics[width=6cm]{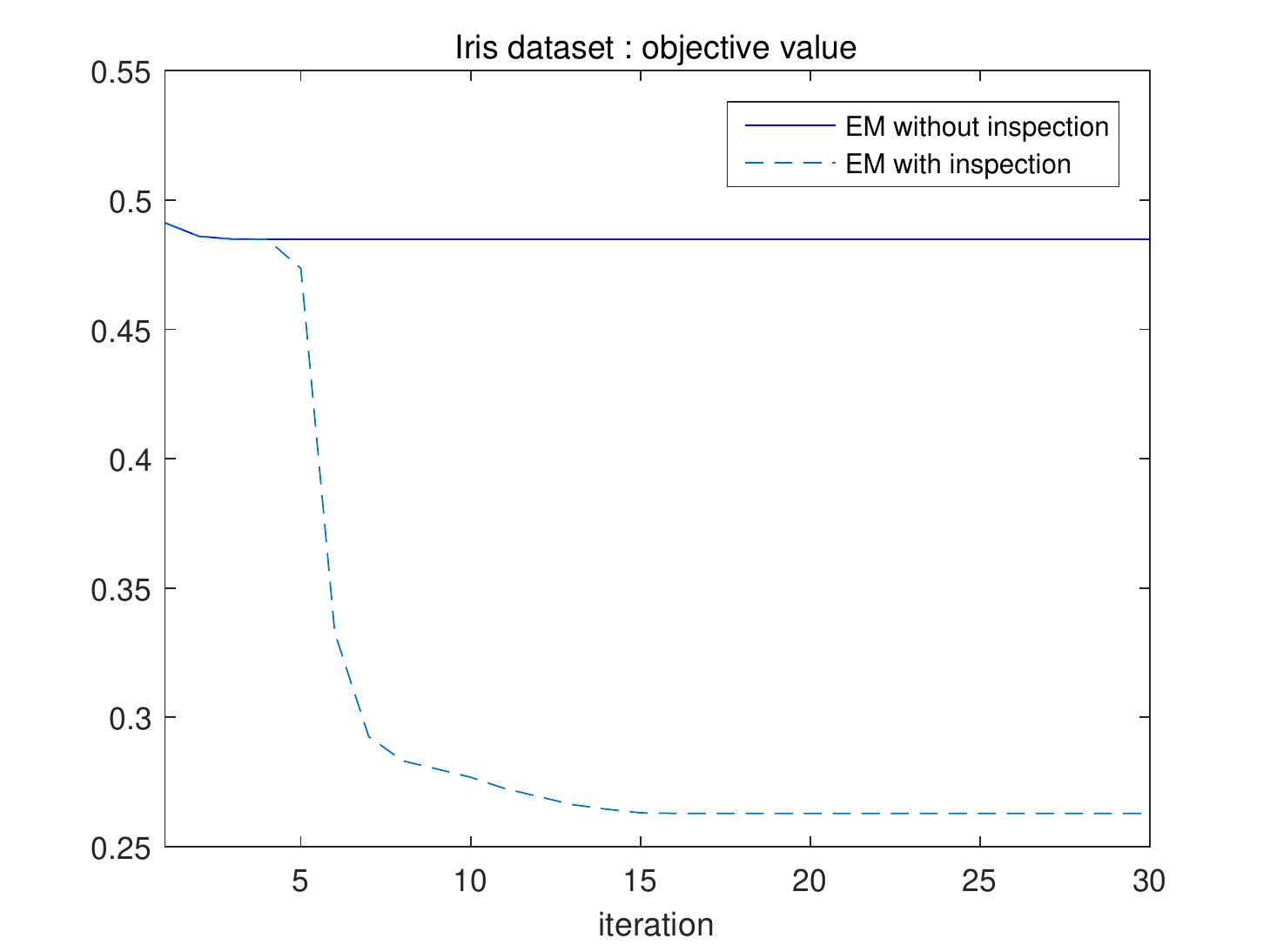} 
\caption{
left: 3-D distribution of Iris data and clustering result in one trial; right: objective value in the iteration of this trial.}
\label{fig: iris_result}
\end{figure}

Among the 500 runs, EM gets stuck at a high objective value 0.48 for 109 times. With the help of inspection, it manages to locate the optimal objective value around 0.263 every time. The average radius-at-escape during the inspections is 2, and the average number of inspections is merely 1.
\subsection{Nonconvex robust linear regression}
In linear regression, we are given a linear model
\[y = \dotp{\vbeta,\vx}+\varepsilon,\]
and the data points $(\vx_1,y_1),(\vx_2,y_2),\ldots,(\vx_n,y_n)$, $y_i\in\RR, \vx_i\in\RR^n$. When there are outliers in the data, robustness is necessary for the regression model. Here we consider Tukey's bisquare loss, which is bounded, nonconvex and defined as:
\begin{align*}
\rho(r)=\left\{\begin{array}{ll}
\frac{r_0^2}{6}\{1-(1-(r/r_0)^2)^3\},  & \text{if } |r|<r_0, \\
\\
\frac{r_0^2}{6},                  & \text{otherwise.}
\end{array}
\right.
\end{align*}
The empirical loss function based on $\rho$ is
\begin{align*}
    l(\vbeta)=\frac{1}{n}\sum\limits_{i=1}^n \rho(y_i-\dotp{\beta, x_i}).
\end{align*}
A commonly used algorithm for this problem is the Iteratively Reweighted Least Squares (IRLS) algorithm ~\cite{fox2002robust}, which may get stuck at a local minimizer. Our Run-and-Inspect Method can help IRLS escape from local minimizers and converge to a global minimizer. Our test uses the model
\[y = 5 + x + \varepsilon,\]
where $\varepsilon$ is noise. We generate $x_i\sim\mathcal{N}(0,1)$, $\varepsilon_i\sim\mathcal{N}(0,0.5), i=1,2,\ldots,20$. We also create 20\% outliers by adding extra noise generated from $\mathcal{N}(0,5)$. And we use Algorithm \ref{algorithm:inspecting process}  with $R=5,dR=0.5,\nu=10^{-3}$. For Tukey's function, $r_0$ is set to be 4.685. The results are shown in Figure \ref{fig:robustregression}.

\begin{figure}[ht]
    \centering
    \includegraphics[width=6cm]{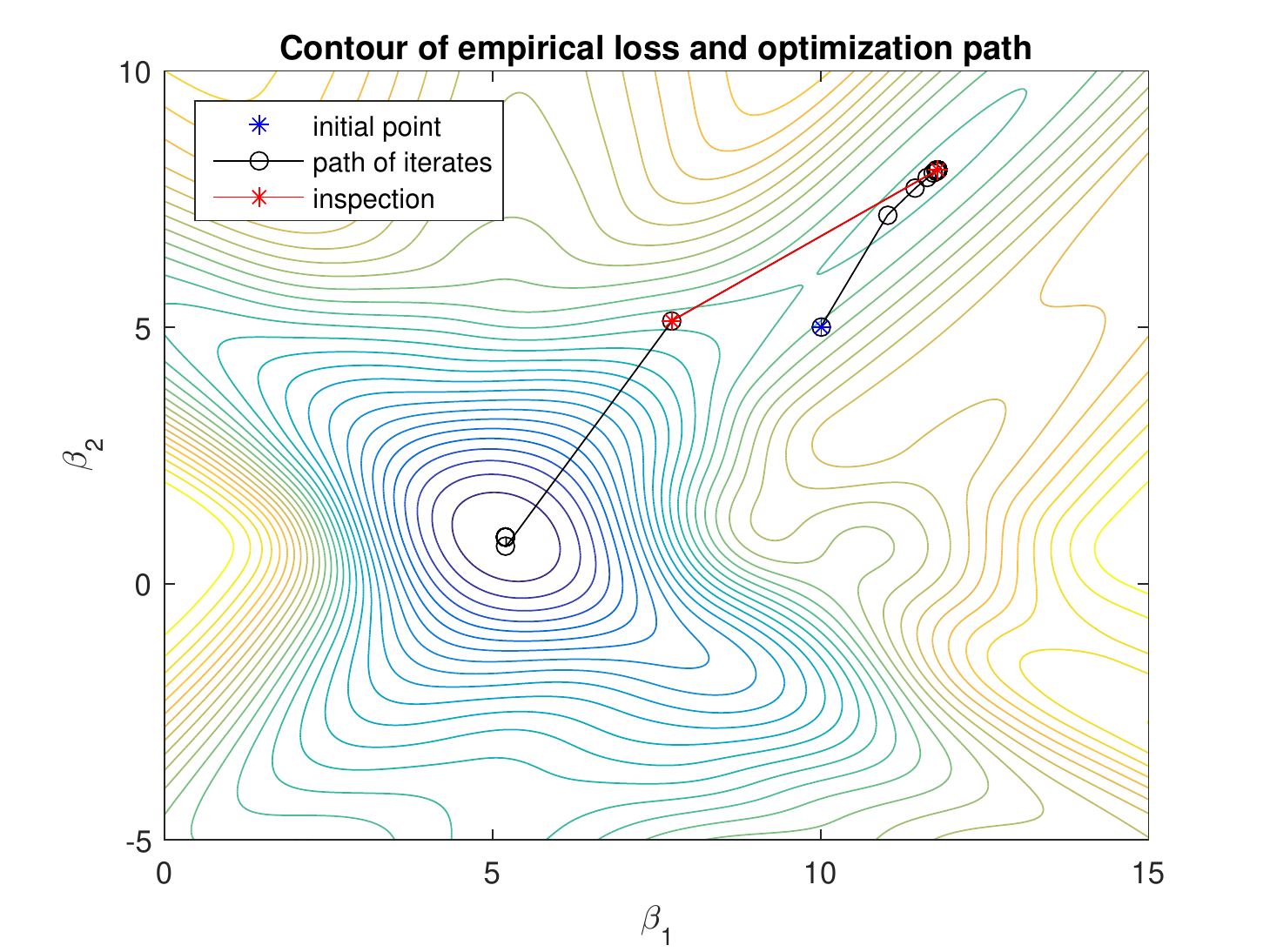}\includegraphics[width=6cm]{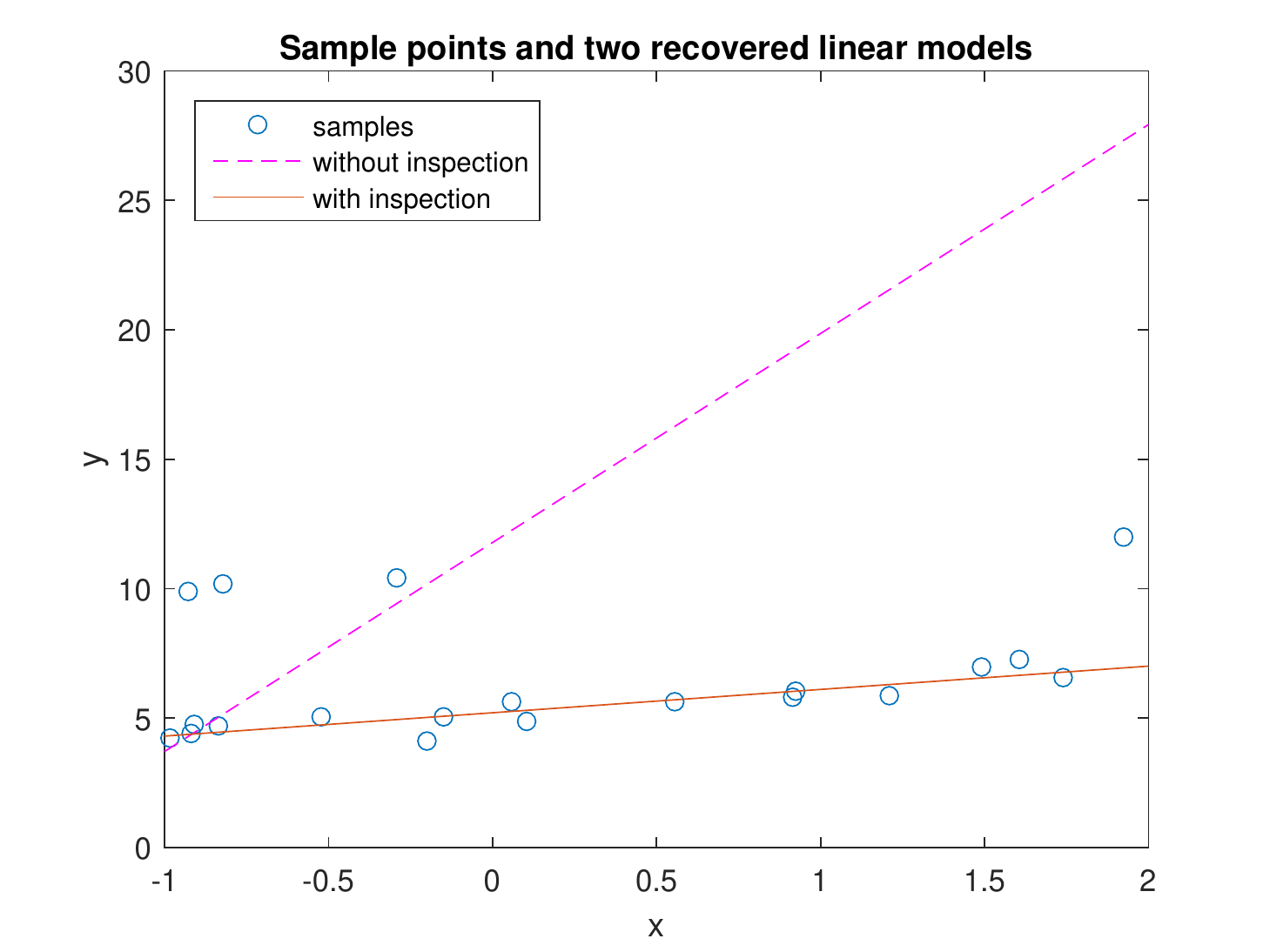}
    \caption{The left picture displays the contour of the empirical loss $l(\vbeta)$ and the path of iterates. Starting from the initial point, IRLS converges to a shallow local minimum. With the help of inspection, it escapes and then converges to the global minimum. The right picture shows linear model obtained by IRLS with (red) and without (magenta) inspection.}
    \label{fig:robustregression}
\end{figure}

\subsection{Nonconvex compressed sensing}
\label{section:compressed sensing}
Given a matrix $A\in\RR^{m\times n}$ $(m<n)$ and a sparse signal $\vx\in\RR^n$, we observe a vector
$$\vb=A\vx.$$
The problem of compressed sensing aims to recover $x$ approximately. Besides $\ell_0$ and $\ell_1$-norm, $\ell_p(0<p<1)$ quasi-norm is often used to induce sparse solutions. Below we use $\ell_{\frac{1}{2}}$ and try to solve the problem
$$\min\limits_{\vx\in\RR^n}\ Q(\vx):=\frac{1}{2}\|A\vx-\vb\|^2+\lambda \|\vx\|_{\frac{1}{2}}^{\frac{1}{2}},$$ 
by cyclic coordinate update. At iteration $k$, it updates the $j$th coordinate, where $j=\mathrm{mod}(k,n)+1$, via
\begin{align}\label{eqn:half thresholding}
x_j^{k+1}&=\argmin\limits_{x_j} Q(x_j,\vx_{-j}^k)\\
&=\argmin\limits_{x_j}\ \frac{1}{2}A_j^TA_jx_j^2+A_j^T(A\vx^k-\vb)x_j+\lambda\sqrt{|x_j|}.
\end{align}
It has been proved in \cite{xu2012l_} that \eqref{eqn:half thresholding} has a closed-form solution. 
Define
\begin{align*}
&B_{j,\mu}(\vx)=x_j-\mu A_j^T(A\vx-\vb),\\
&H_{\lambda,\frac{1}{2}}(z)=\left\{\begin{array}{ll}
\frac{2}{3}z(1+\cos(\frac{2\pi}{3}-\frac{2}{3}\arccos(\frac{\lambda}{4}(\frac{|z|}{3})^{-\frac{3}{2}}))),     &  \text{if }|z|>\frac{^3\sqrt{54}}{4}(2\lambda)^{\frac{2}{3}},\\
0,     & \text{otherwise.}
\end{array}\right.
\end{align*}
Then
\begin{align*}
x_j^{k+1}= H_{\lambda\mu,1/2}(B_{j,\mu}(\vx^k)),
\end{align*}
where $\mu=\|A_j\|^2$.
In our experiments, we choose $m = 25,50,100$ and $n=2m$. The elements of $A$ are generated from $\mathcal{U}(0,\frac{1}{\sqrt{m}})$ i.i.d. The vector $\vx$ has 10\% nonzeros with their values generated from $\mathcal{U}(0.2,0.8)$ i.i.d. Set $b = A\vx$. Here, we apply coordinate descent with inspection (CDI), and compared it with standard coordinate descent (CD) and half thresholding algorithm (\emph{half})~\cite{xu2012l_}. For every pair of $(m,n)$, we choose the parameter $\lambda=0.05$ and run 100 experiments. 
When the iterates stagnate at a local minimizer $\bar{\vx}$, we perform a blockwise inspection with each block consisting of two coordinates. Checking all pairs of two coordinates is expensive and not necessary since $\bar{\vx}$ is sparse. We improve the efficiency by pairing only $i,j$ where $x_i\neq 0, x_j=0$. Similar to previous experiments, we sample points from the outer of the 2D ball toward the inner. We choose $R=0.5$, $\Delta R=0.05$. The results are presented in Table \ref{result:compressedsensing} and Figure \ref{fig:compressedsensing}. CDI shows a significant improvement over its competitors.
\begin{table}[ht]
\centering
\begin{tabular}{c|c|c|c|c|c}
\hline
$n,p$ & algorithm & $a$ & $b$ & $c$ & ave obj \\ \hline\hline
$n=25$ & \emph{half}  & 47.73\% & 2  &  2 & 0.0365 \\ \cline{2-6}
$p=50$ & CD   & 62.40\% & 25 & 27 & 0.0272 \\ \cline{2-6}
       & CDI  & 83.95\% & 65 & 69 & 0.0208 \\ \hline
$n=50$ & \emph{half}  & 46.43\% &  0 &  0 & 0.0736 \\ \cline{2-6}
$p=100$& CD   & 76.39\% & 24 & 32 & 0.0443 \\ \cline{2-6}
       & CDI  & 92.34\% & 57 & 68 & 0.0369 \\ \hline
$n=100$& \emph{half}  & 44.31\% &  0 &  0 & 0.1622 \\ \cline{2-6} 
$p=200$& CD   & 85.97\% & 10 & 18 & 0.0795 \\ \cline{2-6}
       & CDI  & 94.31\% & 54 & 76 & 0.0756 \\ \hline
\end{tabular}
\caption{Statistics of 100 compressed sensing problems solved by three $\ell_{\frac{1}{2}}$ algorithms}
\label{result:compressedsensing}
\flushleft{\small{
\begin{enumerate}
\item \emph{half}: iterative half thresholding; CD: coordinate descent; CDI: CD with inspection.
\item $a$ is the ratio of correctly identified nonzeros to true nonzeros, averaged over the 100 tests (100\% is impossible due to noise and model error); $b$ is the number of tests with all true nonzeros identified; $c$ is the number of tests in which the returned points yield lower objective values than that of the true signal (only model error, no algorithm error). Higher $a,b,c$ are better.
\item ``ave obj'' is the average of the objective values; lower is better. 
\end{enumerate}}}
\end{table}
\begin{figure}[ht]
    \centering
    \includegraphics[width=8cm]{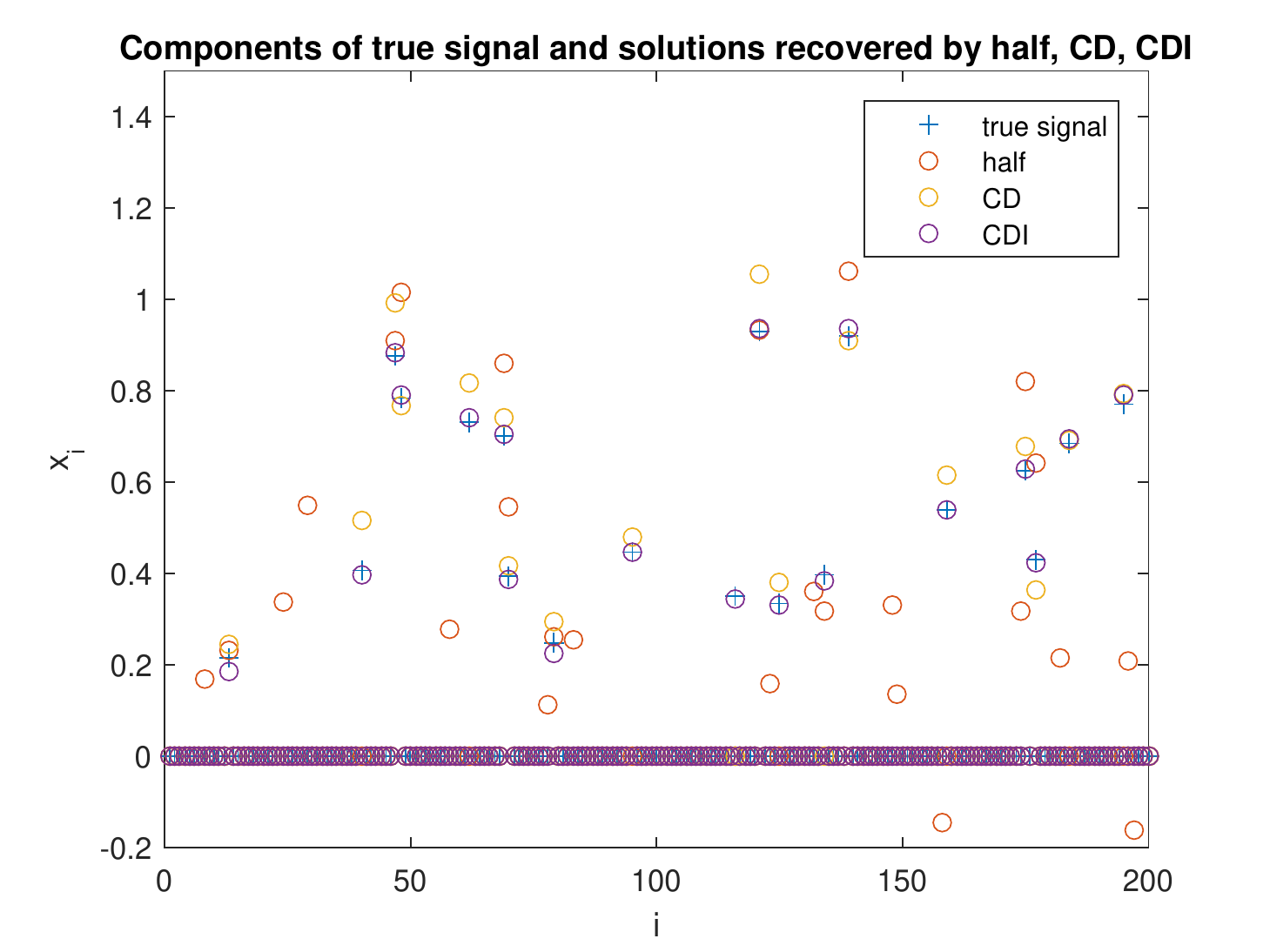}
    \caption{Comparison of the true signal $x$ and signals recovered from \emph{half}, CD, CDI.}
    \flushleft{\small{In one experiment, CDI recovered all positions of nonzeros of $x$, while CD failed to recover $x_{116},x_{134}$. The \emph{half} algorithm just got stuck at a local minizer far from $x$. }}
    \label{fig:compressedsensing}
\end{figure}
\subsection{Nonconvex Sparse Logistic Regression}
Logistic regression is a widely-used model for classification. Usually we are given a set of training data $\{(\vx^{(i)},y^{(i)})\}_{i=1}^N$, where $\vx^{(i)} \in \RR^d $ and $y^{(i)} \in \{0,1\}$. The label $y$ is assumed to satisfy the following conditional distribution:
\begin{align}
\begin{cases}
    p(y=1|\vx;\theta)=\frac{1}{1+\exp(-\vtheta^T\vx)},\\
    p(y=0|\vx;\theta)=\frac{1}{1+\exp(\vtheta^T\vx)},
\end{cases}
\end{align}
where $\vtheta$ is the model parameter.

To learn $\theta$, we minimize the negative log-likelihood function 
\[l(\vtheta) = \sum_{i=1}^N -\log p(y^{(i)}|\vx^{(i)};\vtheta),\]
which is convex and differentiable. When $N$ is relatively small, we need variable selection to avoid over-fitting. 
In this test, we use the minimax concave penalty (MCP)~\cite{zhang2010nearly}:
\begin{equation*}
p_{\lambda,\gamma}^{\text{MCP}}(x)=\begin{cases}
\lambda|x|-\frac{x^2}{2\gamma}, & \text{if~} |x|\leq\gamma\lambda,\\
\frac{1}{2}\gamma\lambda^2, & \text{if~} |x|>\gamma\lambda.
\end{cases}
\end{equation*}
The $\vtheta$-recovery model writes
\[\min_{\vtheta} l(\vtheta)+\beta p_{\lambda,\gamma}^{\text{MCP}}(\vtheta).\]
The penalty $p_{\lambda,\gamma}^{\text{MCP}}$ is proximable with
\begin{align*}
{\rm prox}_p(z)=\begin{cases}
\frac{\gamma}{\gamma-1}S_{\lambda}(z) & \text{if}~|z|\leq\gamma\lambda,\\
z & \text{if}~|z|>\gamma\lambda
\end{cases}
\end{align*}
where $S_\lambda(z)=(|z|-\lambda)_{+}\sign(z)$. 

We apply the prox-linear (PL) algorithm to solve this problem. When it nearly converges, inspection is then applied. We design our experiments according to \cite{Shen2017Nonconvex}: we consider $d=50$ and $N=200$ and assume the true $\theta^*$ has $K$ non-zero entries.
In the training procedure, we generate data from i.i.d. standard Gaussian distribution, and we randomly choose $K$ non-zero elements with i.i.d standard Gaussian distribution to form $\theta^*$.
The labels are generated by $y=1(\vx^T\theta+w\geq 0)$,
where $w$ is sampled according to the Gaussian distribution $\mathcal{N}(0,\epsilon^2I)$. We use PL iteration with and without inspection to recover $\theta$. After that, we generate 1000 random test data points to compute the test error of the $\theta$. We set the parameter $\beta=1.5-0.06\times K$, $\lambda=1$, $\gamma=5$ and the step size 0.5 for PL iteration. For each $K$ and $\epsilon$, we run 100 experiments and calculate the mean and variance of the results. The inspection parameters are $R=5$, $\Delta R=1$, and $\Delta\theta=\pi/10$. The sample points in inspections are similar to those in section \ref{section:compressed sensing}. The results are presented in Table \ref{result:logistic}. The objective values and test errors of PLI, the PL algorithm with inspection, are significantly better than the native PL algorithm. On the other hand, the cost is also 3 -- 6 times as high.
\begin{table}[ht]
\centering
\begin{tabular}{c|c|c|c|c|c|c}
\hline
\multirow{2}{*}{$K,\epsilon$} & \multirow{2}{*}{\begin{tabular}[c]{@{}l@{}}algorithm\end{tabular}} & \multirow{2}{*}{\begin{tabular}[c]{@{}l@{}}average \#.iterations \\ / \#.inspections\end{tabular}} & \multicolumn{2}{c|}{objective value} & \multicolumn{2}{c}{test error} \\ \cline{4-7} 
                              &                                                                                       &                                                                                     & mean              & var         & mean          & var        \\ \hline\hline
$K=5$                         & PL                                                                              & 594                                                                            & 48.0         & 305         & 7.26\%        & 1.55e-03        \\ \cline{2-7} 
$\epsilon=0.01$               & PLI                                                                                  & 3430/11.43                                                                      & 26.8          & 44.9         & 3.79\%        & 5.43e-04        \\ \hline
$K=5$                         & PL                                                                               & 601                                                                            & 52.7          & 409         & 7.81\%        & 1.29e-03        \\ \cline{2-7} 
$\epsilon=0.1$                & PLI                                                                                  & 2280/7.98                                                                       & 33.8          & 51.9         & 4.38\%        & 5.43e-04        \\ \hline
$K=10$                        & PL                                                                               & 1040                                                                            & 43.6          & 87.0         & 8.42\%        & 8.68e-04        \\ \cline{2-7} 
$\epsilon=0.01$               & PLI                                                                                  & 2610/4.78                                                                       & 33.5          & 35.9         & 5.73\%        & 5.73e-04        \\ \hline
$K=10$                        & PL                                                                               & 990                                                                            & 47.5    & 87.2         & 9.41\%        & 9.88e-04        \\ \cline{2-7} 
$\epsilon=0.1$                & PLI                                                                                  & 2370/3.86                                                                       & 36.3          & 40.1         & 5.69\%        & 5.50e-04        \\ \hline
$K=15$                        & PL                                                                               & 1600                                                                            & 36.2          & 54.3         & 7.85\%        & 8.29e-04        \\ \cline{2-7} 
$\epsilon=0.01$               & PLI                                                                                  & 3010/3.21                                                                       & 29.2         & 17.5         & 5.77\%        & 5.20e-04        \\ \hline
$K=15$                        & PL                                                                               & 1570                                                                            & 37.1          & 40.2         & 7.80\%        & 8.30e-04        \\ \cline{2-7} 
$\epsilon=0.1$                & PLI                                                                                  & 2820/2.77                                                                       & 30.7          & 16.0         & 6.66\%        & 4.63e-04        \\ \hline
\end{tabular}
\caption{Sparse logistic regression results of 100 tests. PL is the prox-linear algorithm. PLI is the PL algorithm with inspection. ``var'' is variance.}
\label{result:logistic}
\end{table}

We plot the convergence history of the objective values in one trial and the recovered $\theta$ in Figure \ref{fig: logistic}.
\begin{figure}[ht]
     \centering
     \includegraphics[width=7cm]{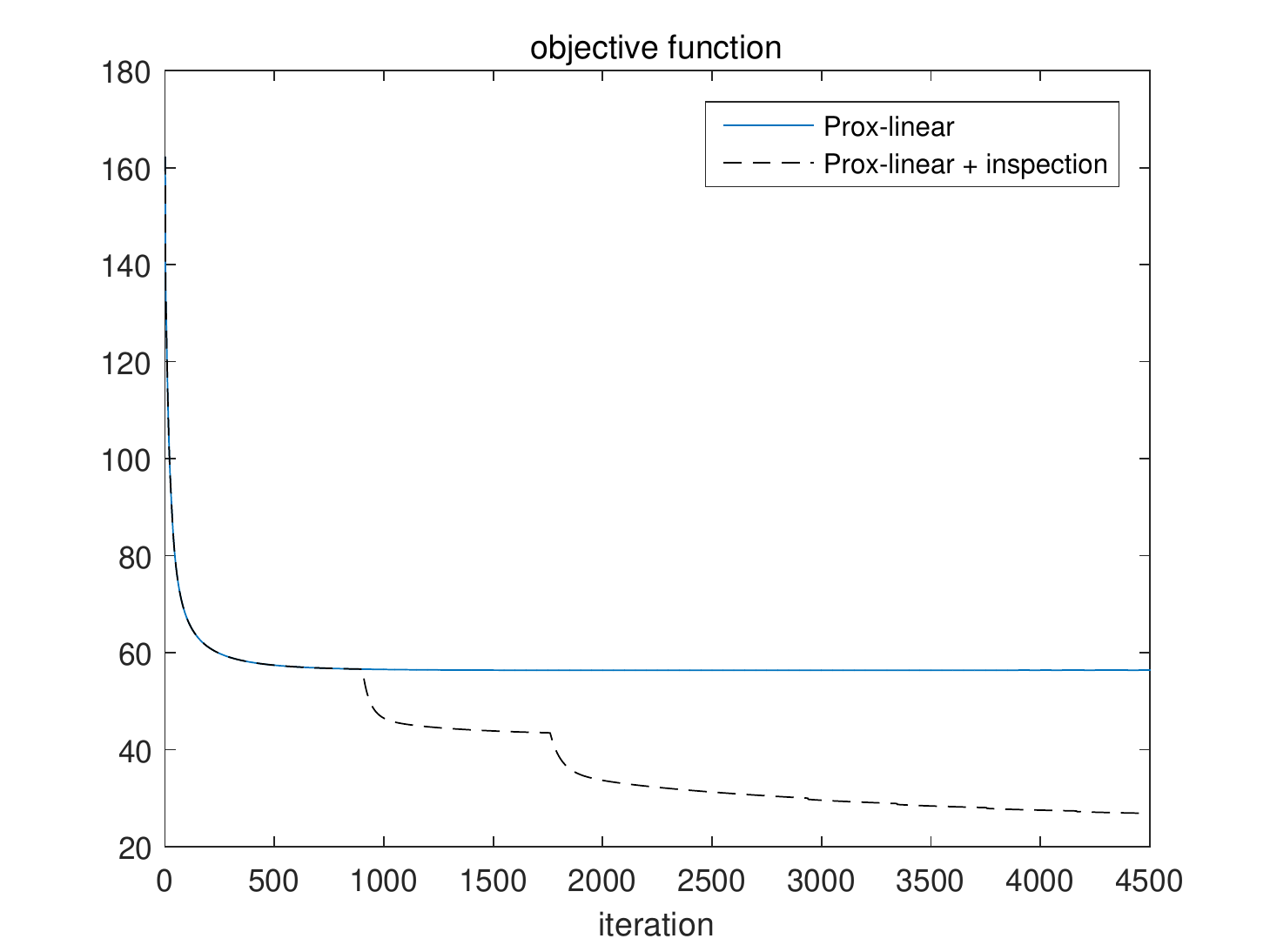} \includegraphics[width=7cm]{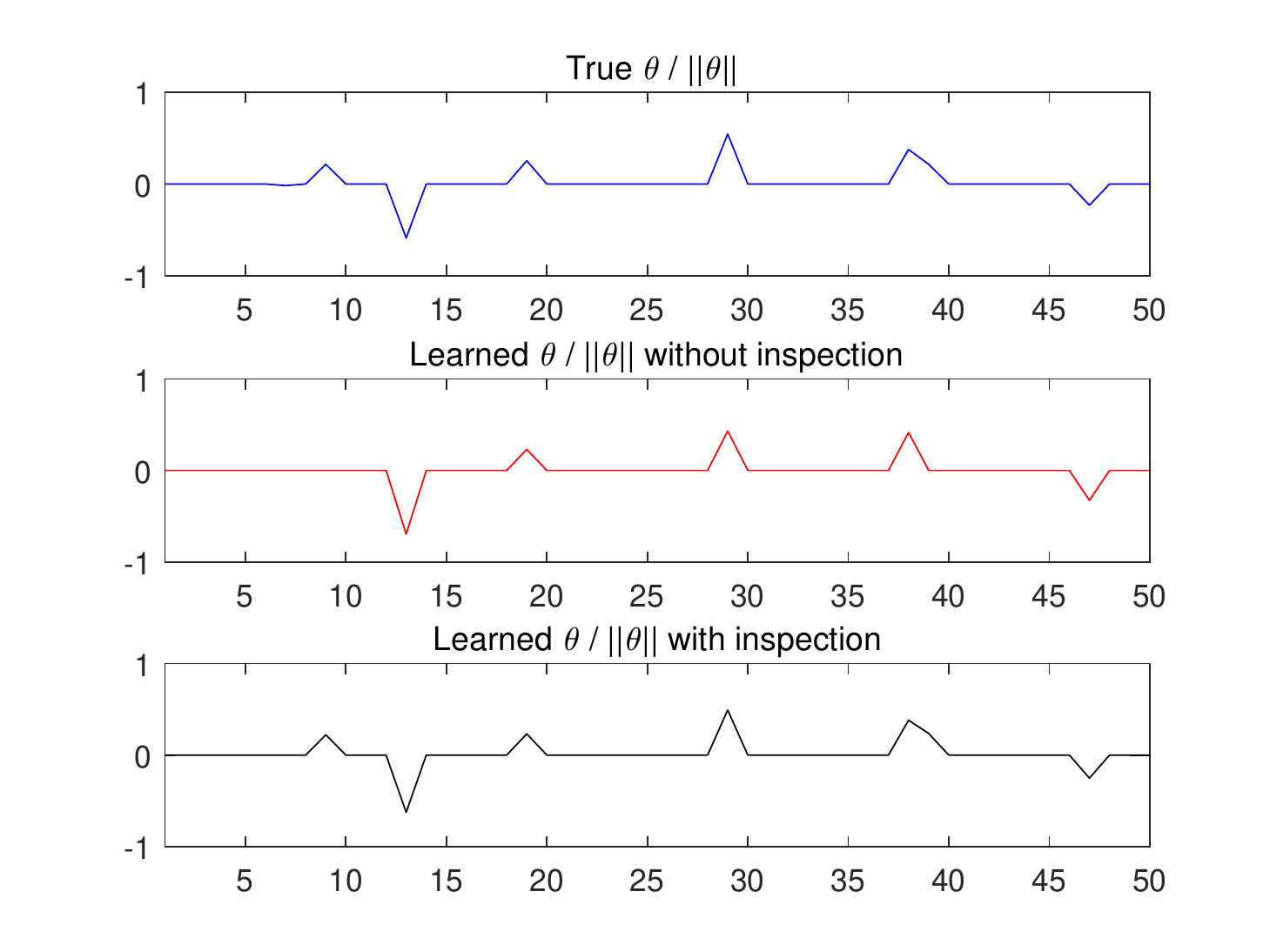}
  
  \caption{Sparse logistic regression result in one trial.}
  \label{fig: logistic}
\end{figure}
It is clear that the inspection works in learning a better $\theta$ by reaching a smaller objective value.
\section{Conclusions}
\label{section conclusions}
In this paper, we have proposed a simple and efficient method for nonconvex optimization, based on our analysis of $R$-local minimizers. The method applies local inspections to escape from local minimizers or verify the current point is an $R$-local minimizer. For a function that can be implicitly decomposed to a smooth, strongly convex function plus a restricted nonconvex functions, our method returns an (approximate) global minimizer. Although some of the tested problems may not possess the assumed decomposition, numerical experiments support the effectiveness of the proposed method.

\renewcommand\refname{Reference}
\bibliographystyle{spmpsci} 
\bibliography{ref}

\end{document}